\documentclass[11pt]{article}

\usepackage[utf8]{inputenc}
\usepackage{amsmath,amssymb,amsthm}
\usepackage{graphicx}
\usepackage{float}
\usepackage{geometry}
\usepackage{tikz-cd}
\usepackage{url}
\usepackage{hyperref}
\long\def\forget#1{}

\newtheorem{theorem}{Theorem}
\newtheorem{proposition}[theorem]{Proposition}
\newtheorem{example}[theorem]{Example}
\newtheorem{remark}[theorem]{Remark}
\newtheorem{definition}[theorem]{Definition}

\title{Motive Theory Hidden\\ in\\ Karajī--Pascal\footnote{Karajī--Pingala--Xian--Khayyam--Tartaglia--Pascal and ..., cf. Section \ref{SectHistory} and Remark \ref{RemarkHistoricalFigures}.} Triangle}
\author{Somayeh Habibi\footnote{School of Mathematics, Institute for Research in Fundamental Sciences (IPM)}}
\date{\today}

\begin{document}

\maketitle

\noindent
\textit{``For in and out, above, about, below,
\\'Tis nothing but a Magic Shadow-show\\Play'd in a Box whose Candle is the Sun,\\Round which we Phantom Figures come and go."\\ ~~~~~~~~Omar Khayyam (1048–1131) }

\section*{Abstract}

These lecture notes are intended as an accessible introduction to some basic ideas of motive theory for readers with limited background in algebraic geometry. Mathematics often reveals unexpected connections between seemingly distant areas. A simple combinatorial identity may encode geometric structures, arithmetic information, or even sophisticated categorical phenomena.

In these notes, we trace a path from elementary counting arguments to Voevodsky's theory of motives. We show how some classical combinatorial identities emerge naturally from geometry, and how motivic decompositions reveal the deeper geometric and arithmetic structures underlying them. Our guiding example is provided by the Karajī--Pascal identity and its $q$-analogue, which link combinatorics and algebraic geometry.

Thus our primary aim of these lecture notes is to demonstrate that some familiar combinatorial identities can provide non-experts with an entry point to some of the basic ideas of motive theory!  Moreover, while introducing the reader to the subject, we also hope to encourage the view that even an elementary mathematical formula may encode a deeper underlying geometric and arithmetic structure. Along the way, the notes offer a gentle introduction to motives through a concrete example rather than through the full technical machinery of modern motive theory.

In the appendix, we briefly introduce adequate equivalence relations on algebraic cycles. In addition, we study a generating series for the higher Chow groups with coefficients in a finite field. This part of the paper assumes additional background.

\tableofcontents

\section{Karajī-Pascal Triangle}\label{SectHistory}

The binomial coefficients are the numbers $b_{r,k}$ that appear in the binomial expansion

\[
(x+y)^r
=
\sum_{k=0}^r
b_{r,k}x^ky^{r-k},
\]
and the binomial theorem asserts that \(b_{r,k}=\binom{r}{k},\) where

\[
\binom{r}{k} = \frac{r!}{k!(r-k)!}.
\]

The known history of binomial theorem goes back to  Euclid (4th century B.C.) where he mentions the formula for $(x+y)^2$. In the 3-rd century
B.C. the Indian lyricist Piṅgala describes a method of arranging two types of syllables to form metres of various lengths and counting them. 
The 10th century Persian mathematician Karajī (953--1029) has written a treatise containing a triangular arrangement of binomial coefficients

\[
\begin{array}{ccccccccc}
&&&&1&&&&\\[4pt]
&&&1&&1&&&\\[4pt]
&&1&&2&&1&&\\[4pt]
&1&&3&&3&&1&\\[4pt]
1&&4&&6&&4&&1
\end{array}
\]
where each entry is obtained recursively as the sum of the two entries immediately above it:
\[
b_{r,k}=b_{r-1,k}+b_{r-1,k-1}.
\]
with boundary values are taken to be \(b_{r,0}=b_{r,r}=1.\)

In addition, he represented the earliest systematic appearance of the mathematical induction, see \cite{Rashed}. Although the original work has not survived, later mathematicians preserved his method. An explicit account appears in the work \emph{al-Bāhir} by al-Samaw'al (12th century), where the results are directly attributed to Karajī. 

\begin{figure}[H]
\centering
\includegraphics[width=0.411\textwidth]{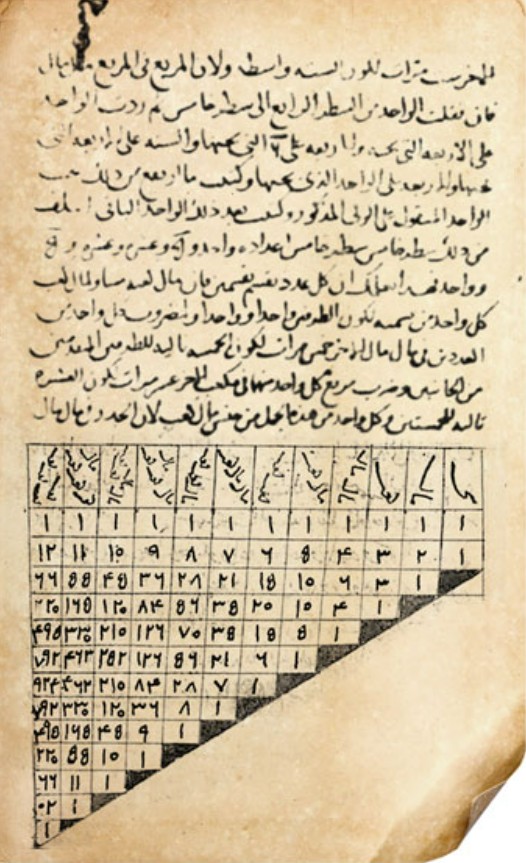}
\caption{Binomial Coefficients, ''al-Bahir fi'l-jabr" by l-Samaw'al (c. 1130 – c. 1180), cf. \cite[page 111]{Ah-Ra}}
\end{figure}

After realizing that the binomial coefficients are the numbers
\[
\binom{r}{k}
=
\frac{r!}{k!(r-k)!},
\qquad 0\le k\le r,
\]
which count the number of ways to choose \(k\) elements from an \(r\)-element set, the above recursive rule can also be rewritten as
\begin{equation}\label{EqKayyam-Pascal}
\binom{r}{k}
=
\binom{r-1}{k-1}
+
\binom{r-1}{k}.
\end{equation}

Karajī's presentation was later reproduced and developed by Omar Khayyám (1048–1131), and the triangle is therefore also known as Khayyám's triangle, who employed in his method for extracting roots of polynomials of degree higher than two.\\

Related developments appeared notably independently in Chinese mathematics. During the thirteenth century, Yang Hui and Chu Shih-Chieh described binomial expansions and triangular coefficient arrays in their mathematical works. Yang Hui explicitly credited these methods to the earlier mathematician Jia Xian from the eleventh century, whose original writings are likewise no longer extant. We refer the reader to \cite{Ah-Ra}, see also \cite{Goss}.

\begin{figure}[H]
\centering
\includegraphics[width=0.45\textwidth]{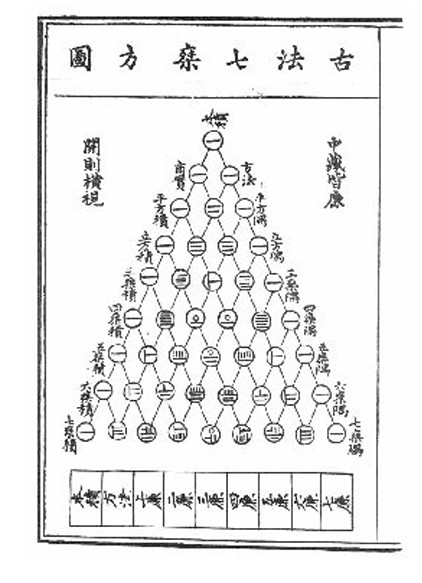}
\caption{Binomial Coefficients from Chu Shih-Chieh's book}
\end{figure}

In Europe, the triangle first appeared in the Arithmetic of Jordanus de Nemore during the thirteenth century.
In the following centuries, Gersonides computed binomial coefficients using their multiplicative formula,
while Petrus Apianus published a variation of triangle in 1527. Michael Stifel reproduced a substantial portion of it in 1544, and Niccolò Tartaglia published six rows in 1556, leading to its Italian name, Tartaglia's triangle.  Gerolamo Cardano subsequently described both the additive and multiplicative rules for constructing the triangle.

\begin{remark}\label{RemarkHistoricalFigures}
Although throughout the remainder of this note we refer to the triangle and identity \eqref{EqKayyam-Pascal} as the \emph{Karajī--Pascal} triangle and the \emph{Karajī--Pascal} identity, as it is clear from the title and the corresponding footnote, we wish to acknowledge the significant contributions of many other mathematicians made over the course of several centuries.
\end{remark}

\bigskip

\section{Lattice Paths and the $q$-Analogue}\label{SectLatticePaths}
The identity (\ref{EqKayyam-Pascal}) admits a simple combinatorial interpretation. To choose an $m$-element subset from a set of size $m+n$ (Note that we made a change of variables: $m:=k$ and $n:=r-k$), fix one distinguished element. Any chosen subset either contains this element or does not contain it. If the distinguished element is absent, we choose all $m$ elements from the remaining $m+n-1$ elements. If the distinguished element is present, we only need to choose $m-1$ additional elements. Adding the two possibilities gives Karajī--Pascal's identity. 

Now consider the integer lattice
\[
\mathbb Z^2=\{(i,j): i,j\in\mathbb Z\}.
\]
We study lattice paths from \((0,0)\) to \((m,n)\) consisting only of
\emph{east} steps \((1,0)\) and \emph{north} steps \((0,1)\). Such a path
can be represented by a word containing \(m\) letters \(E\) (each
corresponding to one east step) and \(n\) letters \(N\) (each corresponding
to one north step). Conversely, every such word determines a unique lattice
path. Therefore, the total number of these paths is
\(
\binom{m+n}{m}.
\)

\begin{figure}[H]
\centering
\includegraphics[width=0.45\textwidth]{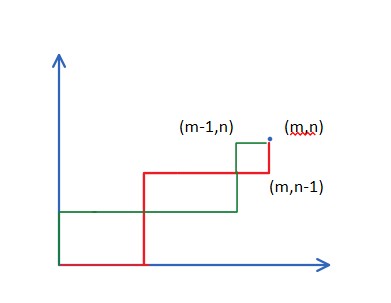}
\caption{each path
can be represented by a word containing \(m\) letters \(E\) and \(n\) letters \(N\)}
\end{figure}

One can reprove Karajī--Pascal's identity by separating these paths into two disjoint classes. Namely, every path arriving at \((m,n)\) must make its final step either \[(m-1,n)\to (m,n)\] or \[(m,n-1)\to (m,n).\]





Now, let us assign a weight to each path. If $A_p$ denotes the area above the path $p$ (and under the line $y=n$), define the weight \(q^{A_p}.\)

Summing over all paths produces a polynomial:

\[
Q(m,n)=\sum_p q^{A_p}.
\]

This weighted counting satisfies the recursion

\[
Q(m,n)=Q(m-1,n)+q^mQ(m,n-1).
\]

If we instead weight the paths according to the area under the paths (weight path $p$ with $q^{U_p}$, where $U_p$ denotes the area under $p$), we get the following alternative recurrence relation

\[
F(m,n)=
F(m,n-1)
+
q^{n}F(m-1,n).
\]

Note that there is the following duality between paths:

\paragraph{Duality between the paths}\label{Duality between the paths}

Let \(p\) be a path from origin to $(m,n)$. Assume it is represented by the word:
\[
w=w_1w_2\cdots w_{m+n},
\qquad
w_i\in\{E,N\}.
\]

Define the dual path \(p^\ast\) by reversing the word:
\[
w^\ast
=
w_{m+n}w_{m+n-1}\cdots w_1.
\]

Note that the area above a path is equal to the area under its dual path. According to this duality phenomena between the paths, we can argue that $Q(m,n)=F(m,n)$.  The resulting polynomial is the Gaussian binomial coefficient, also called the $q$-binomial coefficient:

\[
\genfrac{[}{]}{0pt}{}{m+n}{m}_q.
\]

The $q$-analogue replaces ordinary integers and factorials with their $q$-versions:

\[
[n]_q=\frac{q^n-1}{q-1},
\]

\[
[n]_q!=[n]_q[n-1]_q\cdots[1]_q,
\]

and

\[
\genfrac{[}{]}{0pt}{}{m+n}{m}_q
=
\frac{[m+n]_q!}{[m]_q![n]_q!}.
\]

Note that in the limit $q \to 1$, this recovers the ordinary binomial coefficient.

\bigskip

\noindent
Let's now DO algebraic Geometry! Below we will see how these polynomials interpolate between combinatorics and geometry. Before doing so, however, we first move toward:

\bigskip
\section{Zeta Function}

\subsection{Dynamical Systems and Zeta Functions}


Let $f:X\to X$ be a dynamical system on a topological space $X$. For every positive integer $n$, define

\[
\operatorname{Fix}(f^n)=\{x\in X:f^n(x)=x\}.
\]

Assuming the number of fixed points is finite for every iterate, the Artin--Mazur zeta function is defined by

\[
\zeta(f,t)
=
\exp\left(
\sum_{n=1}^{\infty}
\frac{\#\operatorname{Fix}(f^n)}{n}t^n
\right).
\]

This generating function packages the periodic orbit structure of the dynamical system, see \cite{Ar-Ma}.

\begin{example}\label{Example:The Circle Map} (The Circle Map)

\noindent
Consider the map

\[
f:S^1\to S^1,
\qquad
f(z)=z^m,
\quad m\ge2.
\]
Its iterates satisfy \(f^n(z)=z^{m^n}.\) The fixed points of $f^n$ are solutions of \(z^{m^n}=z,\) which reduces to counting roots of unity. The corresponding zeta function captures this periodic structure. Hence
\[
\#\operatorname{Fix}(f^n)=m^n-1.
\]
Therefore
\[
\log \zeta(f,t)
=
\sum_{n\ge1}\frac{m^n-1}{n}t^n
=
-\log(1-mt)+\log(1-t),
\]
so
\[
\zeta(f,t)=\frac{1-t}{1-mt}.
\]

\end{example}

This dynamical formalism becomes especially powerful when transferred to algebraic geometry. Namely:

\bigskip

\subsection{Hasse--Weil Zeta Functions}

Let $X$ be a quasi-projective variety over the finite field $\mathbb F_q$.

The Hasse--Weil zeta function of $X$ is defined by

\[
Z(X,s)
=
\exp\left(
\sum_{r=1}^{\infty}
\frac{\#X(\mathbb F_{q^r})}{r}(q^{-s})^r
\right).
\]

If we write \(z=q^{-s},\) then the zeta function becomes a generating function encoding the number of points of $X$ over finite field extensions. An important observation is that

\[
\#X(\mathbb F_{q^r})
=
\#\operatorname{Fix}(\sigma^r),
\]

where $\sigma$ denotes the Frobenius morphism on \(X(\overline{\mathbb F}_q).\)

Here $\sigma$ is the $q$-power Frobenius morphism, defined
on geometric points by raising each coordinate to the
$q$-th power:
\[
\sigma(x_0,\ldots,x_N)=(x_0^q,\ldots,x_N^q)
\]
in affine coordinates, or
\[
\sigma([x_0:\cdots:x_N])=[x_0^q:\cdots:x_N^q]
\]
in projective coordinates. Since $X$ is defined over
$\mathbb F_q$, it is a well-defined map over $X$. A point is fixed
by $\sigma^r$ precisely when it is defined over
$\mathbb F_{q^r}$.

Thus Hasse--Weil zeta functions are analogues of Artin--Mazur zeta functions, with the Frobenius dynamic.

\begin{example}\label{ExampleGL_n}
\begin{enumerate}
    \item The points of projective space \(\mathbb P^n\) over $\mathbb F_q $ parametrizes one-dimensional subspaces of \(\mathbb F_q^{\,n+1}\). Since \(\mathbb F_q^{\,n+1}\) has \(q^{n+1}-1\) nonzero vectors, and each line contains exactly \(q-1\) nonzero vectors, we obtain

\(
\#\mathbb P^n(\mathbb F_q)
=
\frac{q^{n+1}-1}{q-1}
=
q^n+q^{n-1}+\cdots+q+1
=
[n+1]_q.
\)

Thus
 $Z(\mathbb{P}^n,z) = exp(\sum_{r=1}^\infty \frac{(q^r)^n+(q^r)^{n-1}+\dots+q^r+1}{r}z^r)=\prod_{i=0}^n\left(\frac{1}{1-q^{i}z}\right)$

\item

To count invertible matrices over \(\mathbb F_q\), choose the columns successively. The first column can be any nonzero vector, giving \(q^n-1\) choices. The second column must be linearly independent from the first, giving \(q^n-q\) choices, and so on. Hence

\[
\#GL_n(\mathbb F_q)
=
(q^n-1)(q^n-q)\cdots(q^n-q^{n-1}).
\]

Factoring out powers of \(q\), one obtains

\[
\#GL_n(\mathbb F_q)
=
q^{n(n-1)/2}(q-1)^n[n]_q!.
\]

Using the definition of the Hasse--Weil zeta function one obtains the product formula

\[
Z(\operatorname{GL}_n/\mathbb F_q,z)
=
\prod_{S\subseteq[n]}
\bigl(1-q^{\,n^2-\sum_{i\in S}i}z\bigr)^{(-1)^{|S|+1}}.
\]

One can observe that for $n=1$ the Frobenius dynamic on $\mathbb G_m$ is the analogues to the circle map dynamic we discussed in Example \ref{Example:The Circle Map}.

\end{enumerate}

\end{example}

\bigskip

\section{Grassmannians and Gaussian Binomial Coefficients}\label{SectGrassmannians and Gaussian Binomial Coefficients}

Now, we return to the $q$-binomial coefficients from a geometric perspective.\\

\subsection{Points of Grassmannian over a Finite Field}\label{SubSectPointsOfGr}

Consider the Grassmannian \(
\text{Gr}_m^n:=Gr(m,m+n)
\). Recall that it is a variety whose points over a field $L$ parametrize $m$-dimensional subspaces of $L^{m+n}$. Over a finite field $\mathbb F_q$, the set \(\text{Gr}_m^n(\mathbb F_q)\) consists of all $m$-dimensional subspaces of \(
\mathbb F_q^{m+n}.\) The group \(GL_{m+n}(\mathbb F_q)\) acts transitively on this set. By analyzing the stabilizer subgroup, one obtains

\[
\#\text{Gr}_m^n(\mathbb F_q)
=
\frac{[m+n]_q!}{[m]_q![n]_q!}
=
\genfrac{[}{]}{0pt}{}{m+n}{m}_q,
\]
see Example \ref{ExampleGL_n}. Thus, Gaussian binomial coefficients count points on Grassmannians over finite fields. This is the first major indication that $q$-combinatorics is fundamentally geometric.

\bigskip

\subsection{Geometric Interpretation Behind the $q$-Karajī--Pascal Identity}\label{SubSectGeoInt}

Fix a one-dimensional subspace \( \ell\subset \mathbb F_q^{m+n}.\) Partition the Grassmannian into two subsets:

\[
\mathcal U
=
\{T\in \text{Gr}_m^n:\dim(T\cap \ell)=0\},
\] and \[
\mathcal Z
=
\{T\in \text{Gr}_m^n:\dim(T\cap \ell)=1\}.
\]
These subsets correspond exactly to the two terms in the $q$-Karajī--Pascal recursion. Geometrically:

\begin{itemize}
    \item[-] $\mathcal U$ is an affine bundle over $\text{Gr}_m^{n-1}$.
    \item[-] $\mathcal Z$ identifies with $\text{Gr}_{m-1}^n$.
\end{itemize}

\begin{figure}[h]
\centering
\begin{tikzpicture}[scale=0.85,>=latex,thick]

\draw (0,0) circle (2);

\filldraw[
    fill=blue!20,
    draw=blue!60!black
]
(-0.45,0) ellipse (1.45 and 1.55);

\node at (-0.45,0) {$\mathcal U$};
\node at (1.25,0) {$\mathcal Z$};

\draw[->]
(-1.25,-1.0) -- (-4.0,-2.5)
node[midway,above,sloped]
{$\mathbb A^m_{\mathbb F_q}\text{-bdl}$};

\node at (-4.6,-2.9) {$\mathrm{Gr}_m^{\,n-1}$};

\draw[->]
(1.55,0.55) -- (4.0,-1.8)
node[midway,above,sloped]
{$\cong$};

\node at (4.6,-2.2) {$\mathrm{Gr}_{m-1}^{\,n}$};

\end{tikzpicture}
\caption{Stratification of the Grassmannian variety}
\end{figure}
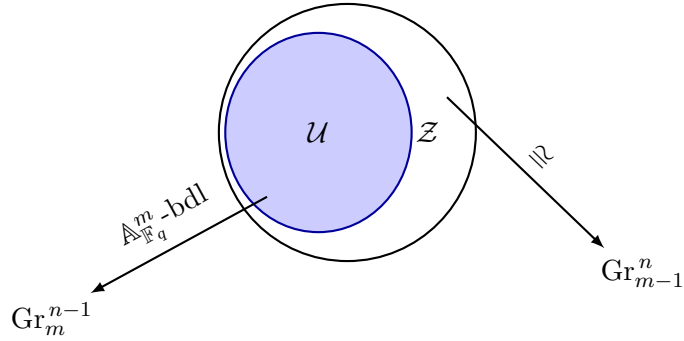

Counting points over $\mathbb F_q$ yields

\[
\genfrac{[}{]}{0pt}{}{m+n}{m}_q
=
\genfrac{[}{]}{0pt}{}{m+n-1}{m-1}_q
+
q^m
\genfrac{[}{]}{0pt}{}{m+n-1}{m}_q
.
\]
What originally appeared as a combinatorial recursion is therefore a reflection of a geometric stratification.

\bigskip

\begin{remark}\label{RemDualFormula}
Dually, one can fix a hyperplane $H$ (instead of a line). Then can divide the set of $m$-dimensional subspaces into two disjoint subsets: First those subspaces who are contained in $H$, and second those subspaces who are not contained in $H$. It is not hard to show that counting the points of Grassmannian in this way, gives the second q-binomial analog of Karajī-Pascal identity:
    \[
\genfrac{[}{]}{0pt}{}{m+n}{m}_q
=
q^n \genfrac{[}{]}{0pt}{}{m+n-1}{m-1}_q
+
\genfrac{[}{]}{0pt}{}{m+n-1}{m}_q
.
\]
\end{remark}

\section{Weil Conjectures and the Birth of Motives}

\subsection{Cycles and equivalence relations}

One of the fundamental problems of algebraic geometry is to study algebraic
subvarieties of a given variety up to a natural notion of equivalence.
For a variety \(X\), let \(Z^{p}(X)\) denote the free abelian group generated
by the irreducible closed subvarieties of codimension \(p\); its elements are
called \emph{algebraic cycles} of codimension \(p\).

The appropriate equivalence relation is \emph{rational equivalence}, which
identifies cycles that occur as fibers of a family parameterized by the
projective line. Intuitively, two cycles are rationally equivalent if one can
be continuously deformed into the other through an algebraic family over
\(\mathbb{P}^{1}\). The resulting quotient is the Chow group.

More precisely, let \(z^{p}(X,1)\) denote the group of codimension \(p\)
cycles on \(X\times\mathbb{P}^{1}\) that meet the fibers over
\(0,\infty\in\mathbb{P}^{1}\) properly. The boundary map

\[
\partial=i_{0}^{*}-i_{\infty}^{*}:
z^{p}(X,1)\longrightarrow Z^{p}(X),
\]
is defined by taking the difference of the restrictions to the fibers over
\(0\) and \(\infty\). The Chow group of codimension \(p\) is then

\[
CH^{p}(X)
=
\operatorname{coker}
\left(
\partial:
z^{p}(X,1)\longrightarrow Z^{p}(X)
\right).
\]
Equivalently,
\(
CH^{p}(X)=Z^{p}(X)/\sim_\text{rational equivalence}.
\)

When \(X\) is a smooth algebraic variety, the intersection product of cycles
descends to rational equivalence, endowing the graded group
\[
CH^{*}(X)=\bigoplus_{p\ge0}CH^{p}(X)
\]
with the structure of a graded commutative ring, called the \emph{Chow ring}; see \cite{Ful}.
This ring encodes the intersection theory of algebraic cycles and is a
fundamental tool in algebraic geometry.

The theory of Chow groups originated in the work of Wei-Liang Chow in the
1950s, who introduced a rigorous algebraic theory of intersection of cycles.
It was subsequently developed by Grothendieck and his school, where Chow
groups became one of the basic invariants of algebraic varieties and a
cornerstone of modern intersection theory. They play a central role in the
study of characteristic classes, correspondences, and algebraic cycles, and
form the algebraic analogue of singular homology in topology.

A major breakthrough came with Bloch's introduction of \emph{higher Chow
groups}, which refine the classical Chow groups by organizing algebraic cycles
into a chain complex, cf. \cite{Bloch}. In particular, the ordinary Chow groups appear as the
degree-zero homology,
\[
CH^{p}(X)=H_{0}\bigl(z^{p}(X,\bullet)\bigr),
\]
while the higher homology groups \(CH^{p}(X,n)\) provide a geometric model for
motivic cohomology and connect algebraic cycles with algebraic \(K\)-theory. See appendix \ref{app:SubSectHigherChow}.

One may think of the parameter space \(\mathbb{P}^{1}\) as an algebraic
``model of time.''  A cycle on \(X\times\mathbb{P}^{1}\) describes a
one-parameter family of algebraic cycles on \(X\), and the two distinguished
points \(0,\infty\in\mathbb{P}^{1}\) represent the ``initial'' and ``final''
states of the family.  The boundary map

\[
\partial=i_0^*-i_\infty^*
\]

measures the change of the family between these two moments.  Declaring such
differences to be zero means that two cycles are identified whenever one can
evolve into the other through an algebraic family parameterized by
\(\mathbb{P}^{1}\).  This is precisely the notion of \emph{rational
equivalence}, and the resulting quotient is the Chow group.

The choice of \(\mathbb{P}^{1}\) is not arbitrary.  It is the simplest complete
rational curve and serves as the basic parameter space for rational motions.
From this viewpoint, \(\mathbb{P}^{1}\) plays the role of a universal clock:
it provides the simplest algebraic notion of a one-dimensional time along
which cycles may vary.

One can enlarge this notion by allowing more general parameter spaces.  If two
cycles are connected by a family parameterized by an arbitrary connected
algebraic variety \(T\), they are said to be \emph{algebraically equivalent}.
Thus rational equivalence uses the specific ``time model''
\(\mathbb{P}^{1}\), whereas algebraic equivalence permits many different
models of time, namely arbitrary connected algebraic varieties.  Since every
rational equivalence is an algebraic equivalence, there is a natural
implication
\[
\text{rational equivalence}
\;\Longrightarrow\;
\text{algebraic equivalence},
\]
but the converse is false. A classical way to see this is through the Néron--Severi group
of an abelian variety \(A\). Recall that
\[
\operatorname{NS}(A)
=
\operatorname{Pic}(A)/\operatorname{Pic}^{0}(A),
\]
where \(\operatorname{Pic}(A)\) is the group of line bundles on \(A\) and
\(\operatorname{Pic}^{0}(A)\) consists of those line bundles that are
algebraically equivalent to zero. Since \(CH^{1}(A)\cong \operatorname{Pic}(A),\)
passing from rational equivalence to algebraic equivalence amounts to
quotienting \(CH^{1}(A)\) by \(\operatorname{Pic}^{0}(A)\), yielding the
Néron--Severi group \({NS}(A).\) The subgroup \(\operatorname{Pic}^{0}(A)\) is a positive-dimensional
abelian variety, whereas \(\operatorname{NS}(A)\) is a finitely generated
abelian group. Thus there exist many divisor classes that are algebraically
equivalent to zero but are not rationally equivalent to zero. Consequently,
rational equivalence is strictly stronger than algebraic equivalence in
general.

\subsection{Birth of Motives}

The theory of motives emerged from attempts to understand deep analogies between topology, algebraic geometry, and arithmetic. One of the central motivations came from the celebrated Weil conjectures, formulated by André Weil in 1949 \cite{Weil}.

Weil proposed that algebraic varieties over finite fields should satisfy certain properties that are indicated by their topological nature. His philosophy may be summarized by the famous slogan

\begin{quote}
``Topology knows arithmetic.''
\end{quote}

More precisely, Weil observed that the number of points of a variety over finite fields, seemed to behave as though they were governed by hidden topological invariants! Let us recall the statement. 
Let \(X\) be a smooth projective variety of dimension \(d\) over the finite field \(\mathbb F_q\). Then the conjectures state:

\begin{enumerate}
    \item \textbf{Rationality:} The zeta function is a rational function:
    
    \[
    Z(X,z)\in \mathbb Q(z).
    \]

    \item \textbf{Functional Equation:}  The zeta function satisfies a symmetry replacing \(z\) and \(q^{-d}z^{-1}\):
    
    \[
    Z\!\left(X,q^{-d}z^{-1}\right)
    =
    \pm q^{\chi d/2} z^{\chi} Z(X,z),
    \]
    
    where \(\chi\) is the Euler characteristic of \(X\).

    \item \textbf{Betti Number:} The zeta function admits a factorization
    
    \[
    Z(X,z)
    =
    \frac{
    P_1(z)P_3(z)\cdots P_{2d-1}(z)
    }{
    P_0(z)P_2(z)\cdots P_{2d}(z)
    },
    \]

    and the degrees of the polynomials \(P_i\) equal the Betti numbers \(b_i\).

    \item \textbf{Riemann Hypothesis:} The reciprocal roots of \(P_i(z)\) have absolute value \(
    q^{i/2}.\)
    
\end{enumerate}

Here is a short history of the proof. Grothendieck developed étale cohomology and proved the cohomological interpretation and functional equation in 1965, see \cite{GrothI} and \cite{GrothII}, although the rationality part was already proved by Dwork using $p$-adic methods \cite{Dwork} in 1960. Finally Deligne proved the analogue of the Riemann hypothesis \cite{Deligne} in 1974, see also \cite{Kat}.  Below, let us briefly explain how motive theory was born, not as a separate chapter, but as the story that ties these historical milestones into a single unifying vision. History tells us what happened. Story tells us what it meant.    

Weil himself proved several important cases of his conjectures, including the cases of algebraic curves and abelian varieties. His method strongly suggested that a deeper geometric mechanism was at work. A crucial observation was the following: 

``if one could construct a sufficiently good cohomology theory for algebraic varieties, then most of the Weil conjectures would follow formally from standard properties of cohomology, analogous to those used in algebraic topology.

However, serious obstacles soon appeared. Jean-Pierre Serre observed that there could be no satisfactory cohomology theory with coefficients in $\mathbb Q$ possessing all the desired properties in positive characteristic. In particular, the classical tools of singular cohomology were unavailable for varieties defined over finite fields.

This difficulty led Grothendieck to develop a remarkable collection of new cohomology theories during the 1960s, including $\ell$-adic cohomology. These theories successfully captured many of the expected properties of a universal cohomological theory, despite depending on the auxiliary prime $\ell$.

Grothendieck noticed something striking: although these different cohomology theories were constructed in very different ways, they behaved as though they were shadows or realizations of a single underlying object. This observation led to his revolutionary idea of \emph{motives}, as he describes in a letter to Serre in 1964. He says:

\begin{quote}
 I will say that something is a
``motive" over $k$ if it looks like the $\ell$-adic cohomology group of an algebraic scheme over
$k$, but is considered as being independent of $\ell$, with its “integral structure”, or let us say
for the moment its “$\mathbb Q$” structure, coming from the theory of algebraic cycles. The sad
truth is that for the moment I do not know how to define the abelian category of motives,
even though I am beginning to have a rather precise yoga for this category, let us call it
$M(k)$. For example, for any prime $\ell\neq p$, there is an exact functor $T_\ell$ from $M(k)$ into the
category of finite-dimensional vector spaces over $\mathbb Q_\ell$ on which the pro-group $(Gal(\overline{k_i}/k_i))_i$
acts, where $k_i$ runs over subextensions of finite type of $k$ and $\overline{k_i}$ is the algebraic closure of
$k_i$ in $k$; this functor is faithful but not, of course, fully faithful.
\end{quote}

Manin explains the origin of the word by quoting Herbert Read, the English poet and art critic, writing about the French Impressionist painter Paul C\'ezanne:

His method of painting was first to choose his `motif' --- a landscape, a person to be portrayed, a still-life; then to bring into being his visual apprehension of this motif; and in this process to lose nothing of the vital intensity that the motif possessed in its actual existence. For further discussion on Grothendieck's and Cézanne's notions of ``motif", see \cite{Kejian}.

\begin{figure}[H]
\centering
\includegraphics[width=0.45\textwidth]{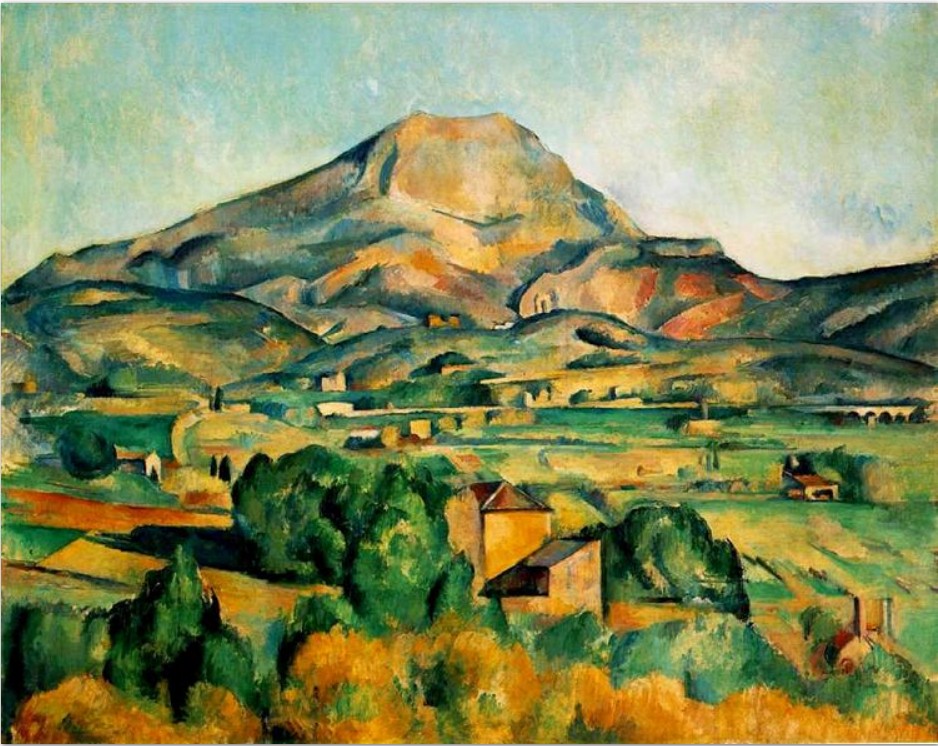}
\caption{Mont Sainte-Victoire, Paul C\'ezanne c.1895}
\end{figure}

So, the word ``motive'' accordingly was intended to suggest the common source or hidden structure behind all cohomological realizations. Rather than viewing each cohomology theory independently, Grothendieck envisioned a universal category from which all such theories could be derived through suitable realization functors.

One of Grothendieck's deepest insights was that the classical notion of a morphism between varieties was not sufficient for capturing these hidden structures. Instead, he proposed enlarging the morphisms themselves. Starting from the category of smooth projective varieties, Grothendieck's construction replaces morphisms between smooth
projective varieties by algebraic correspondences modulo a
chosen adequate equivalence relation; See appendix \ref{app:SubSectAdequateEquivalence}. In this framework, geometric information is encoded not merely by maps between varieties, but by algebraic cycles relating them. 
 Applying the construction of pure motives to these relations
gives several motivic categories. E.g. Chow motives use rational
equivalence; homological motives depend on the chosen
cohomology theory; numerical motives use intersection
numbers.  

Grothendieck formulated the \emph{standard conjectures} as a
foundation for his theory of pure motives. In the conclusion
of \cite{GrothIII}, he explicitly describes this theory in
terms of algebraic cycles modulo numerical equivalence,
rather than rational equivalence.
Standard conjectures were intended to provide the foundational properties of this motivic category. If true, these conjectures would imply that the category of (numerical) motives with rational coefficients is Tannakian semisimple. Such a structure would place the category of motives in a role similar to that of the category of representations of a reductive group. He says:

\bigskip

\textit{``The proof of the two standard conjectures would yield results going considerably further than Weil’s conjectures. They would form the basis of the so-called “theory of motives” which is a systematic theory of “arithmetic properties” of algebraic varieties, as embodied in their groups of classes of cycles for numerical equivalence. We have at present only a very small part of this theory in dimension one, as contained in the theory of abelian varieties. Alongside the problem of resolution of singularities, the proof of the standard conjectures seems to me to be the most urgent task in algebraic geometry."}

See last paragraph of \cite{GrothIII}.


Although many aspects of Grothendieck's original vision remain conjectural, his ideas fundamentally reshaped modern algebraic geometry and laid the foundations for later developments, including Voevodsky's triangulated categories of motives. For further details in this direction we refer the interested reader to \cite{Milne}, \cite{EKM}, \cite{KleimanI}, \cite{KleimanII}, \cite{ManinII}, \cite{Beilinson} and \cite{Manin}.




As Kleiman explains in \cite{KleimanII}, Grothendieck's
standard conjectures imply that the category of numerical
motives with rational coefficients is semisimple abelian.
The underlying argument goes back to Weil's work on curves
and its extension to higher dimensions by Serre.
Kleiman also recalls that he heard Grothendieck describe
this consequence as a ``miracle'' on several occasions.

Remarkably, Jannsen subsequently established this
semisimplicity without assuming the standard conjectures
\cite{Jan}. More precisely, he proved that, over any field
$k$ and with coefficients in a field of characteristic
zero, the category of pure motives modulo an adequate
equivalence relation is semisimple abelian if and only if
the equivalence relation is numerical equivalence.
The significance of this result is emphasized by Murre,
Nagel, and Peters, who describe it as ``undoubtedly the
most important progress on motives after Grothendieck''
\cite[Section 3.2, Remark 1]{MNP}\footnote{the quotation continues: but we are still far away from the standard conjectures }.
In particular, the category of Chow motives is not
semisimple abelian in general. 

If the K\"unneth projectors are algebraic,
then the category of numerical motives with rational
coefficients becomes Tannakian after modifying its
commutativity constraint according to cohomological
parity \cite{Jan}. This does not, however, imply that
the category is neutral Tannakian over $\mathbb Q$.
Moreover, for a given Weil cohomology theory to induce
a fibre functor on numerical motives, homological and
numerical equivalence must coincide for that theory;
see \cite[\S 4.4]{KahnCycles}.

\subsection{The Category of Correspondences}

The categories of correspondences and Chow motives defined
below use rational equivalence and integral coefficients.
Their versions with rational coefficients are obtained by
replacing every Chow group $CH^r(-)$ by
$CH^r(-)_{\mathbb Q}:=CH^r(-)\otimes_{\mathbb Z}\mathbb Q$.

Let \(k\) be a field. The category of correspondences,
denoted by \(\mathrm{Corr}(k)\), is defined as follows.

\begin{itemize}
\item[-] The objects are smooth projective varieties over \(k\).

\item[-] For two objects \(X\) and \(Y\), the group of morphisms from \(X\) to
\(Y\) is
\[
\operatorname{Hom}_{\mathrm{Corr}(k)}(X,Y)
:=
CH^{\dim X}(X\times Y).
\]
Its elements are called \emph{correspondences} from \(X\) to \(Y\).

\item[-] If
\[
\alpha\in CH^{\dim X}(X\times Y),
\qquad
\beta\in CH^{\dim Y}(Y\times Z),
\]
their composition is defined by
\(\beta\circ\alpha=(p_{13})_*\bigl(p_{12}^{*}\alpha\cdot p_{23}^{*}\beta\bigr),\)
where
\[
p_{12},\,p_{23},\,p_{13}
:
X\times Y\times Z
\longrightarrow
X\times Y,\,
Y\times Z,\,
X\times Z
\]
are the natural projections.

\item[-] The identity morphism of an object \(X\) is the diagonal class
\[
\operatorname{id}_X
=
[\Delta_X]
\in
CH^{\dim X}(X\times X).
\]
\end{itemize}

The category of Chow motives is then obtained from this category.

\subsection{The category of Chow motives}

The category of effective Chow motives,
denoted by \(\operatorname{CHM}^{\mathrm{eff}}(k),
\) is the pseudo-abelian (Karoubian) completion of \(\mathrm{Corr}(k)\).

\begin{itemize}
\item[-] The objects are pairs \(
(X,p),\) where \(X\) is a smooth projective variety and \(p\in CH^{\dim X}(X\times X)\) is an idempotent correspondence, i.e. \(p\circ p=p.\)

\item[-] The morphisms \((X,p)\longrightarrow(Y,q)\) are correspondences \(\alpha\in CH^{\dim X}(X\times Y)\) such that
\(\alpha=q\circ\alpha\circ p.\)
Equivalently,
\[
\operatorname{Hom}\bigl((X,p),(Y,q)\bigr)
=
q\circ
CH^{\dim X}(X\times Y)
\circ p.
\]
\end{itemize}

The category of Chow motives,
denoted by \(
\operatorname{CHM}(k),
\) is obtained from
\(\operatorname{CHM}^{\mathrm{eff}}(k)\)
by adjoining Tate twists. Its objects are triples \((X,p,n),\) where \(X\) is a smooth projective variety, \(p\) is an idempotent projector, and \(n\in\mathbb Z\). The morphisms are given by

\[
\operatorname{Hom}_{\operatorname{CHM}(k)}
\bigl((X,p,n),(Y,q,m)\bigr)
=
q\circ
CH^{\dim X+m-n}(X\times Y)
\circ p.
\]

The Chow motive associated to a smooth projective variety \(X\) is

\[
M_{CH}(X):=(X,\Delta_X,0),
\]

where \(\Delta_X\) denotes the class of the diagonal in
\(CH^{\dim X}(X\times X)\).

\bigskip

\section{Voevodsky's Contribution and Triangulated Categories of Motives}

\subsection{Derived Categories and Triangulated Categories}

The notion of a triangulated category was introduced by Grothendieck and Verdier around 1963 in the context of homological algebra and derived categories. Namely, these categories arose naturally from attempts to formalize and abstract the structural properties of the derived category of an abelian category.

To understand the motivation, recall that in classical homological algebra one studies objects such as modules, sheaves, or chain complexes through exact sequences. An exact sequence

\[
0 \to A \to B \to C \to 0
\]
encodes the idea that the object $B$ is an extension of $C$ by $A$, in a precise algebraic sense.

However, when one passes to derived categories, ordinary exact sequences are no longer the fundamental structural objects. Instead, they are replaced by \emph{distinguished triangles}

\[
A \longrightarrow B \longrightarrow C \longrightarrow A[1].
\]

Here $A[1]$ models the shift of a complex. The shift operation reflects the homological grading inherent in derived categories.

Distinguished triangles generalize exact sequences while behaving more naturally in the derived setting. They encode information about mapping cones, extensions, and homological relations between complexes.



Triangulated categories provide a flexible framework for studying geometric and cohomological phenomena where ordinary abelian categories are insufficient. They allow one to systematically handle:

\begin{itemize}
    \item[-] complexes of sheaves,
    \item[-] derived functors,
    \item[-] localization procedures,
    \item[-] homotopy-theoretic constructions,
    \item[-] and duality theories.
\end{itemize}

This abstraction became especially important in algebraic geometry, representation theory, and topology. In motive theory, triangulated categories play a central role because, as we will see, motivic constructions naturally involve complexes, homotopies, and localization phenomena. Triangulated categories provide a natural setting for such constructions, while
retaining the expected connection with an abelian theory
of mixed motives.


Voevodsky's triangulated categories of motives are therefore part of a broader mathematical movement in which derived and homotopical methods replace classical rigid algebraic structures with more flexible categorical frameworks.

\subsection{Voevodsky's Motivic Categories}
Grothendieck's original vision of motives aimed at constructing an abelian category that would serve as a universal receptacle for all cohomology theories. While this framework worked well, at least conceptually, for smooth projective varieties, major difficulties arose when one attempted to extend the theory to non-smooth quasi-projective varieties.

Extending the theory from smooth projective varieties to
more general varieties leads to mixed motives, where
extensions between pure objects become essential.
A suitable abelian category of mixed motives with
rational coefficients remains conjectural in general.

Voevodsky constructed triangulated categories in which
localization, homotopy invariance, and motivic cohomology
can be studied without first constructing this abelian
category \cite{VSF}. This construction does not show that
the abelian theory is unnecessary or impossible.
Rather, the expected abelian category of mixed motives
should arise as the heart of a suitable motivic
$t$-structure on the triangulated category of geometric
motives with rational coefficients; see
\cite[\S4.4]{KahnCycles}.


Beginning in the 1990s, Voevodsky, together with several collaborators, constructed a family of \emph{tensor triangulated} categories of motives, including

\[
DM_{gm}^{eff}(k),
\qquad
DM_{gm}(k),
\qquad
DM_-(k),
\qquad \dots
\]
These categories are now central objects in the $\mathbb A^1$-homotopy theory of schemes.
\noindent
Here:


\begin{itemize}
    \item $DM_{gm}^{eff}(k)$ denotes the category of effective geometric motives, constructed from the bounded homotopy category of finite correspondences by imposing homotopy invariance and Mayer--Vietoris relations and then taking the pseudo-abelian completion. We recall that category of finite correspondences, denoted by $\mathrm{Cor}_k$,
has smooth $k$-schemes as objects, and morphisms from $X$ to $Y$ are
finite correspondences, namely finite $\mathbb{Z}$-linear combinations
of integral closed subschemes of $X\times Y$ that are finite over $X$ and
surjective over an irreducible component of $X$. Composition is defined
by the pullback--intersection--pushforward construction.
    
    \item $DM_{gm}(k)$ is obtained from $DM_{gm}^{eff}(k)$ by formally inverting the Tate motive $\mathbb{Z}(1)$, thereby allowing arbitrary Tate twists.

    \item $DM_-^{\mathrm{eff}}(k,R)$ denotes the category of
effective motivic complexes. For a perfect field $k$,
it is the full triangulated subcategory
\[
DM_-^{\mathrm{eff}}(k,R)
\subset
D^-\bigl(\operatorname{Shv}_{\mathrm{Nis}}^{\mathrm{tr}}(k,R)\bigr)
\]
consisting of bounded-above complexes whose cohomology
sheaves are homotopy invariant. Here
$\operatorname{Shv}_{\mathrm{Nis}}^{\mathrm{tr}}(k,R)$
is the category of Nisnevich sheaves with transfers
of $R$-modules on smooth $k$-schemes.
There is a fully faithful tensor triangulated embedding
\[
DM_{gm}^{\mathrm{eff}}(k,R)
\hookrightarrow
DM_-^{\mathrm{eff}}(k,R);
\]
see \cite[Chapter~5, \S3.1 and Theorem~3.2.6]{VSF}.

    \item $DM_-(k)$ is a larger triangulated derived category of bounded-above complexes of Nisnevich sheaves with transfers whose cohomology sheaves are homotopy invariant, containing geometric motives as a full subcategory.
\end{itemize}

\begin{remark}
    
The superscript $\mathrm{eff}$ is essential:
these effective categories admit nonnegative Tate twists,
but the Tate object has not been made tensor-invertible.
The stabilized category of geometric motives is
\[
DM_{gm}(k,R)
=
DM_{gm}^{\mathrm{eff}}(k,R)[R(1)^{-1}],
\]
in which arbitrary integer Tate twists are allowed.
Thus the bounded-above Nisnevich-sheaf description above
concerns $DM_-^{\mathrm{eff}}(k,R)$, rather than the
stabilized category $DM_{gm}(k,R)$.

\end{remark}

There are functors \(M(-)\) from the category of schemes \(Sch_k\) to either of the above motivic categories, and similarly, compactly supported version \(M^c(-)\). Over a perfect field of positive characteristic, in particular
over a finite field, motives with compact support of smooth
varieties can be defined with rational coefficients by
duality, using de Jong's alterations; see
\cite[Appendix~B, especially Definition~B.3]{HuberKahn}.
Kelly subsequently established the general compact-support
formalism for separated schemes of finite type with
$\mathbb Z[1/p]$-coefficients, using Gabber's refinement
of alterations, where $p=\operatorname{char}(k)$
\cite{Kelly}. Whenever we invoke this general formalism
in positive characteristic, we therefore use coefficients
in $\mathbb Z[1/p]$ or a specified algebra over it.

These functors behave in many ways like generalized cohomology theories. In particular, motivic categories satisfy analogues of several fundamental constructions from algebraic topology, including:

\begin{itemize}
    
    \item Mayer--Vietoris sequences,
For a scheme of finite type $X$ over $k$ and an open
covering $X = U \cup V$ of $X$ one has a canonical distinguished triangle
$$M(U\cap V)\to  M(U)\oplus M(V) \to M(X) \to M(U\cap V)[1]
$$
    \item Homotopy invariance,

For a scheme of finite type $X$ over $k$ the morphism
$M(X \times \mathbb A_k^1) \to M(X)$ is an isomorphism. For motives with compact support we have $M^c(X \times \mathbb A_k^1) \cong M^c(X)(1)[2]$. Note that the unit object of our tensor structure is $\mathbb Z:=M(Spec (k))$. For any smooth scheme $X$ over $k$ the morphism $X \to Spec(k)$
gives us a morphism in $DM_{gm}^{eff}(k)$ of the form $M(X) \to \mathbb Z$. Let $\tilde M(X)$ be defined by the distinguished triangle $\tilde M(X)\to M(X) \to \mathbb Z$ and set $\mathbb Z(1):=\tilde M(\mathbb P^1)[-2]$ and $M(X)(n):=M(X)\otimes \mathbb Z(1)^{\otimes n}$.
    
    \item K\"unneth formulas,
 For schemes of finite type $X$ and $Y$ over $k$ one has a canonical
isomorphism $M(X\times Y) = M(X)\otimes M(Y)$.

    \item Localization triangles,

For a scheme $X$ of finite type over $k$ and a closed subscheme $Z$ in $X$ one has a canonical distinguished triangle 

\begin{equation}\label{EqLocalizingTriangle}
M^c(Z)\to M^c(X)\to M^c(X\setminus Z) \to M^c(Z)[1]
\end{equation}

    \item Duality phenomena

    For any smooth equidimensional scheme $X$ of dimension $d$ over $k$ there is a canonical isomorphism
    
    $$M(X)^\vee = M^c(X)(-d)[-2d]$$

\end{itemize}

Thus motives provide a framework in which geometric decompositions, cohomological relations, and arithmetic information can all be studied simultaneously.



Voevodsky's approach, see \cite{MV}, dramatically expanded the scope of motive theory and transformed it into a workable and highly influential area of modern mathematics. His work ultimately led to major breakthroughs, including the proof of the Milnor conjecture, and the proof of the Bloch–Kato conjecture (the norm residue isomorphism theorem), using a number of innovative results of Markus Rost, see \cite{VoeI} and \cite{VoeII}.


\begin{remark}
For a perfect field $k$, cycle classes and duality define a
fully faithful tensor functor
\[
\iota:\operatorname{CHM}(k)\hookrightarrow DM_{gm}(k),
\qquad
\iota\bigl(M_{CH}(X)(n)\bigr)=M(X)^\vee(n)[2n].
\]
Its action on correspondences is given by the canonical
identification
\[
CH^{\dim X}(X\times Y)
\cong
\operatorname{Hom}_{DM_{gm}(k)}
\bigl(M(X)^\vee,M(Y)^\vee\bigr),
\]
and it extends to projectors by taking the corresponding direct
summands. Here $X,Y$ are smooth projective and
$M(X)^\vee\simeq M(X)(-d)[-2d]$ for $d=\dim X$.

This uses our \emph{cohomological convention}, in which
$M_{CH}(-)$ is contravariant on varieties. In the
\emph{homological convention}, correspondences lie in
$CH_{\dim X}(X\times Y)=CH^{\dim Y}(X\times Y)$,
the functor $M_{CH}^{\mathrm{hom}}(-)$ is covariant, and its
comparison sends $M_{CH}^{\mathrm{hom}}(X)(n)$ to
$M(X)(n)[2n]$.
\end{remark}

\section{The Motivic Karajī--Pascal Identity}

We now return to the Karajī--Pascal identity from the perspective of motive theory. We will see what originally appeared as a simple combinatorial recursion turns out to reflect a geometric decomposition inside the category of motives.

\subsection{The Decomposition}\label{SubSectDecomposition}
The key observation is that the Grassmannian admits a natural stratification which gives rise to a motivic decomposition mirrors the recursive structure of the Gaussian binomial coefficients, see Section \ref{SectGrassmannians and Gaussian Binomial Coefficients}.

\begin{proposition}\label{PropositionMotivicKP}
We have the following decomposition
\begin{equation}\label{EqMotivicKhayyam-Pascal}
M(\text{Gr}_m^n)
=
M(\text{Gr}_{m-1}^n)
\oplus
M(\text{Gr}_m^{n-1})(m)[2m].
\end{equation}

in Voevodsky's triangulated category of motives $DM_{gm}(k)$.
\end{proposition}
\forget{
\begin{proof}

In Voevodsky's triangulated category of motives, the stratification explained in Subsection \ref{SubSectGeoInt}, gives via the localization triangle

\begin{equation}\label{EqMotivicKhayyam-PascalI}
M^c(\mathcal Z) \to M^c(\text{Gr}_m^n)
\to M^c(\mathcal U)
\to M^c(\mathcal Z)[1].
\end{equation}
Since $\mathcal Z$ and $\text{Gr}_m^n$ are projective we have $M^c(\mathcal Z)=M(\mathcal Z)$ and $M^c(\text{Gr}_m^n)=M(\text{Gr}_m^n)$, moreover, since $\mathcal U$ is a fiber bundle over $Gr_m^{n-1}$, we have $$M^c(\mathcal U)=M(\text{Gr}_m^{n-1})(m)[2m].$$ 
The closure of the graph of $p : \mathcal U \to \text{Gr}_m^{n-1}$ in $\text{Gr}_m^{n}\times \text{Gr}_m^{n-1}$ defines a cycle in $CH_{mn}(\text{Gr}_m^{n}\times \text{Gr}_m^{n-1})$ and gives a splitting of the above triangle (\ref{EqMotivicKhayyam-PascalI}), thus

\begin{equation*}
M(\text{Gr}_m^n)
=
M(\text{Gr}_{m-1}^n)
\oplus
M(\text{Gr}_m^{n-1})(m)[2m].
\end{equation*}

\end{proof}

}

\begin{proof}
Assume $m,n\geq1$. Fix an $(m+n)$-dimensional
$k$-vector space $V$ and a line $\ell\subset V$. Let
\[
\mathcal{Z}:=\{T\in \text{Gr}_m^n:\ell\subset T\},
\qquad
\mathcal{U}:=\text{Gr}_m^n\setminus \mathcal{Z}.
\]
The closed subvariety $\mathcal{Z}$ is naturally isomorphic to
$\mathrm{Gr}_{m-1}^n$: a subspace $T$ containing $\ell$
corresponds to the $(m-1)$-dimensional subspace
$T/\ell\subset V/\ell$.
\forget{
We first explain the geometry of the open complement.
Projection to $V/\ell$ defines a morphism
\[
p:U\longrightarrow
B:=\mathrm{Gr}_m^{n-1}.
\]
Indeed, if $T\cap\ell=0$, then its image in $V/\ell$
has dimension $m$. Choose a splitting
$V\simeq\ell\oplus(V/\ell)$. For an $m$-dimensional
subspace $W\subset V/\ell$, every lift belonging to $U$
is the graph of a unique linear map
\[
W\longrightarrow\ell.
\]
Thus $p$ identifies $U$ with the total space of the
rank-$m$ vector bundle
$\mathcal S^\vee\otimes_k\ell$ on $B$, where
$\mathcal S$ is the tautological subbundle.
In particular, affine-bundle invariance gives}

Look at subsection \ref{SubSectGeoInt}. Projection to $V/\ell$, defines a morphism $p:\mathcal{U}\longrightarrow \mathrm{Gr}_m^{n-1}$. By homotopy invariance, we have
\begin{equation}\label{AffBunIso}
M^c(\mathcal{U})\simeq M^c(\mathrm{Gr}_m^{n-1})(m)[2m]
          =M(\mathrm{Gr}_m^{n-1})(m)[2m].
\end{equation}

Write $i:\mathcal{Z}\hookrightarrow \text{Gr}_m^n$ and $j:\mathcal{U}\hookrightarrow \text{Gr}_m^n$
for the inclusions. Since $\text{Gr}_m^n$ and $\mathcal{Z}$ are projective,
the compact-support localization triangle becomes
\[
M(\mathcal{Z})\xrightarrow{i_*}M(\text{Gr}_m^n)
\xrightarrow{\rho}M(\mathrm{Gr}_m^{n-1})(m)[2m]
\longrightarrow M(\mathcal{Z})[1],
\]
where $\rho$ is the restriction
$M(\text{Gr}_m^n)\to M^c(\mathcal{U})$ followed by the
affine-bundle isomorphism, see equation (\ref{AffBunIso}). To obtain a direct-sum
decomposition, we must construct a section of $\rho$.

Let
\[
\Gamma:=\overline{\Gamma_p}\subset \text{Gr}_m^n\times \mathrm{Gr}_m^{n-1}
\]
be the closure of the graph of $p$. \forget{ Since
$\dim \mathcal{U}=\dim G=mn$, its class belongs to
\[
[\Gamma]\in CH_{mn}(\text{Gr}_m^n\times \mathrm{Gr}_m^{n-1};R).
\]
After interchanging the factors, this is a class in} This defines a cycle in:
\[
CH_{\dim \mathrm{Gr}_m^{n-1}+m}(\mathrm{Gr}_m^{n-1}\times \text{Gr}_m^n;R),
\qquad
\dim \mathrm{Gr}_m^{n-1}=m(n-1).
\]
Note that
\[
\operatorname{Hom}\bigl(M(\mathrm{Gr}_m^{n-1})(m)[2m],M(\text{Gr}_m^n)\bigr)
\simeq CH_{\dim \mathrm{Gr}_m^{n-1}+m}(\mathrm{Gr}_m^{n-1}\times \text{Gr}_m^n;R), 
\]
see \cite{VSF} proposition 2.1.4. Consequently, the transposed graph closure defines
a morphism
\[
s:M(\mathrm{Gr}_m^{n-1})(m)[2m]\longrightarrow M(\text{Gr}_m^n).
\]

We claim that $\rho\circ s=\operatorname{id}$.
The essential observation is that restriction of
$\Gamma$ to $\mathcal{U}\times \mathrm{Gr}_m^{n-1}$ is precisely $\Gamma_p$:
\[
(j\times\operatorname{id}_{\mathrm{Gr}_m^{n-1}})^*[\Gamma]=[\Gamma_p].
\]
The transposed graph of the smooth morphism $p$
represents its flat pullback
\[
p^*:M(\mathrm{Gr}_m^{n-1})(m)[2m]\longrightarrow M^c(\mathcal{U}),
\]
which is the affine-bundle isomorphism. Compatibility
of correspondence actions with restriction to open
subschemes therefore gives
\[
j^*\circ s=p^*.
\]
Since $\rho=(p^*)^{-1}\circ j^*$, we obtain
\[
\rho\circ s
=(p^*)^{-1}\circ j^*\circ s
=(p^*)^{-1}\circ p^*
=\operatorname{id}.
\]

\forget{For clarity, this identity can also be checked on
Chow groups. Let $T$ be any smooth projective
$k$-variety, put $p_T=\operatorname{id}_T\times p$,
and let $\beta\in CH_r(T\times \mathrm{Gr}_m^{n-1};R)$.
The graph-closure correspondence sends $\beta$ to
a class in $CH_{r+m}(T\times G;R)$ whose restriction
to $T\times \mathcal{U}$ is
\[
p_T^*\beta\in CH_{r+m}(T\times \mathcal{U};R).
\]
This follows from proper--flat base change and the
graph formula for flat pullback. Applying the inverse
of affine-bundle pullback recovers $\beta$.
Thus $\rho s$ acts as the identity for every such
$T$ and every $r$; equivalently, Manin's identity
principle gives the same identity of motives.}

The localization triangle therefore splits, yielding
\[
M(\text{Gr}_m^n)\simeq M(\mathcal{Z})\oplus M(\mathrm{Gr}_m^{n-1})(m)[2m].
\]
Substituting the descriptions of $\mathcal{Z}$ and $\mathrm{Gr}_m^{n-1}$ gives
\[
M(\mathrm{Gr}_m^n)
\simeq
M(\mathrm{Gr}_{m-1}^n)
\oplus
M(\mathrm{Gr}_m^{n-1})(m)[2m],
\]
as required.
\end{proof}

As a remarkable feature of motives, one may see that different mathematical expressions can be seen as different realizations of the same motivic decomposition.

\subsection{Étale Realization and the $q$-Karajī--Pascal Identity}

Here we assume $\ell\neq p$. Applying the étale realization functor converts motives into $\ell$-adic cohomological data. Computing the trace of Frobenius on the $\ell$-adic realization gives the counting points over finite fields. Thus the decomposition in Proposition \ref{PropositionMotivicKP} produces the classical $q$-analogue of Karajī--Pascal identity:

\[
\genfrac{[}{]}{0pt}{}{m+n}{m}_q
=
\genfrac{[}{]}{0pt}{}{m+n-1}{m-1}_q
+
q^m
\genfrac{[}{]}{0pt}{}{m+n-1}{m}_q.
\]

Thus the Gaussian binomial recursion is not merely a combinatorial coincidence; it reflects an actual decomposition in algebraic geometry.


\subsection{Betti Realization and Partition Identities}

If one instead applies Betti realization to the decomposition in Proposition \ref{PropositionMotivicKP}, the motivic decomposition yields topological information about Grassmannians, and then according to the Weil conjectures, one obtains the following identity involving partition functions(find further explanation below):

\[
p(n,m,k)
=
p(n,m-1,k)
+
p(n-1,m,k-m),
\]

where $p(n,m,k)$ denotes the number of partitions of $k$ into at most $m$ parts, each part less than or equal to $n$.

This identity admits a natural combinatorial interpretation:

\begin{itemize}
    \item either we have less than $m$ parts,
    \item or one can remove one from each part, transforming the problem to partitions of $k-m$ into at most $m$ parts each part less than or equal to $n-1$.
\end{itemize}

From the motivic viewpoint, however, this recursion is no longer isolated combinatorics; it arises from geometric and cohomological structures associated with Grassmannians!
\begin{remark}
   Note that we have a dual equality \( p(n,m,k) =
p(n,m-1,k-n)
+
p(n-1,m,k),
\)
 which can be derived from another decomposition of Grassmannians, you will see in the next subsections, see theorem \ref{PoincareGrassmannianDecomposition}.
\end{remark}

The motivic decomposition in Proposition
\ref{PropositionMotivicKP} also has a natural
combinatorial interpretation through the Schubert
cell decomposition of Grassmannians.

Recall that the Schubert cells of
$\mathrm{Gr}_m^n$ are indexed by partitions
$\lambda=(\lambda_1,\ldots,\lambda_m)$ satisfying
\[
n\geq\lambda_1\geq\cdots\geq\lambda_m\geq0,
\]
or equivalently, by Young diagrams contained in
an $m\times n$ rectangle. The cell corresponding
to $\lambda$ is isomorphic to the affine space
$\mathbb A^{|\lambda|}$, where
$|\lambda|=\lambda_1+\cdots+\lambda_m$.

In particular, the odd Betti numbers vanish, and
the even Betti numbers are given by
\[
\dim_{\mathbb Q}H^{2k}
(\mathrm{Gr}_m^n,\mathbb Q)
=p(n,m,k),
\]
where $p(n,m,k)$ denotes the number of partitions
of $k$ into at most $m$ parts, each bounded by $n$.
Consequently, the Poincar\'e polynomial is
\[
P_q(\mathrm{Gr}_m^n)
=\sum_{k\geq0}p(n,m,k)q^k
=\sum_{\lambda\subseteq n^m}q^{|\lambda|},
\]
where $q$ is initially a formal variable such that its power records
half of the cohomological degree.

Applying Betti realization to the decomposition in
Proposition \ref{PropositionMotivicKP}, we obtain
\[
P_q(\mathrm{Gr}_m^n)
=
P_q(\mathrm{Gr}_{m-1}^n)
+
q^m P_q(\mathrm{Gr}_m^{n-1}).
\]
Here the Tate twist $(m)$ and the cohomological
shift $[2m]$ play distinct roles: the former
records the Tate structure, while the latter
shifts the cohomological degree by $2m$.
Thus, the corresponding summand contributes
a factor $q^m$ to the Poincar\'e polynomial.

Comparing coefficients of $q^k$ gives
\[
p(n,m,k)
=
p(n,m-1,k)
+
p(n-1,m,k-m).
\]
\forget{This identity also has a direct combinatorial
interpretation: a partition counted by $p(n,m,k)$
either has fewer than $m$ nonzero parts, or has
exactly $m$ nonzero parts. In the latter case,
subtracting one from each part produces a
partition of $k-m$ into at most $m$ parts,
each bounded by $n-1$.}

The same polynomial admits an arithmetic
interpretation. If $q$ is now the cardinality
of a finite field $\mathbb F_q$, each Schubert
cell $\mathbb A^{|\lambda|}$ has exactly
$q^{|\lambda|}$ rational points. Therefore,
\[
\#\mathrm{Gr}_m^n(\mathbb F_q)
=
\sum_{\lambda\subseteq n^m}q^{|\lambda|}
=
P_q(\mathrm{Gr}_m^n).
\]
Thus, the recursion arising from the motivic
decomposition simultaneously describes the
Betti numbers and the finite-field point counts
of Grassmannians.

\bigskip

\subsection{Zeta Functions and Motivic Factorization}

The motivic decomposition also induces relations among Hasse--Weil zeta functions.

Indeed, the decomposition of motives implies the factorization

\[
Z(\text{Gr}_m^n,z)
=
Z(\text{Gr}_{m-1}^n,z)
\cdot
Z(\text{Gr}_m^{n-1},q^m z).
\]
This identity reflects how arithmetic information decomposes according to the motivic splitting. The shifted argument $q^m z$ corresponds to the Tate twist appearing in the motivic decomposition. The Tate twist modify the weights of Frobenius eigenvalues and therefore alter the zeta function accordingly.

\begin{remark}
    The Hasse-Weil zeta function counts points over finite fields. Note that there is the notion of the motivic zeta function that encodes the entire cohomological footprint of a variety across all weights simultaneously. It is the generating function not of numbers, but of motives themselves. We are not going to discuss this in this note, the interested reader can refer to \cite{Kah}.
\end{remark}

\bigskip

\subsection{Euler Characteristics and Classical Karajī--Pascal Identity}

Another realization of the motivic decomposition is obtained by taking Euler characteristics.

Applying the Euler characteristic to both sides of the motivic identity yields

\[
\chi(\text{Gr}_m^n)
=
\chi(\text{Gr}_m^{n-1})
+
\chi(\text{Gr}_{m-1}^n).
\]

Since the Euler characteristic of the Grassmannian equals the corresponding binomial coefficient, one recovers the classical Karajī--Pascal identity:

\begin{equation}
\binom{m+n}{m}
=
\binom{m+n-1}{m}
+
\binom{m+n-1}{m-1}.
\end{equation}

Thus we have seen so far that the ordinary Karajī--Pascal recursion, the $q$-binomial identity, partition recursions, and zeta-function factorizations all emerge as various realizations of the same motivic decomposition. This is a simple demonstration of the unifying power of motives. But one can go further:

\bigskip

\subsection{Poincaré Duality}

The motivic decomposition of the Grassmannian not only explains the recursive structure of Gaussian binomial coefficients, but also their remarkable symmetry. This symmetry is a manifestation of Poincaré duality.
Let us first explain the following:

Let \(X\) be a scheme of finite type over $\mathbb{Z}$. We say that \(X\) is
\emph{polynomial count} if there exists a polynomial \(P_X(y)\in \mathbb Z[y]\), such that for every finite field \(\mathbb F_q\), the number of elements of \( X(\mathbb F_q)\) equals \(P_X(q).\) Let \(P_X(y)=\sum_{k=0}^N a_k y^k.\) The Hasse--Weil zeta function of \(X\) over \(\mathbb F_q\) is

\[
Z(X,z)
=
\exp\!\left(
\sum_{r\ge1}
\frac{\sharp X(\mathbb F_{q^r})}{r}z^r
\right).
\]

Now the definition of the zeta function gives

\[
\begin{aligned}
\log Z(X,z)
=
\sum_{r\ge1}
\frac{P_X(q^r)}{r}z^r\\
=
\sum_{k=0}^N
a_k
\sum_{r\ge1}
\frac{(q^k z)^r}{r}.
\end{aligned}
\]

Using the identity \(\sum_{r=1}^\infty \frac{x^r}{r}
=
-\log(1-x),\) we obtain

\[
\log Z(X,z)
=
\sum_{k=0}^N -a_k \log(1-q^k z).
\]

Exponentiating yields

\[
Z(X,z)
=
\prod_{k=0}^N
(1-q^k z)^{-a_k}.
\]

Therefore:
\begin{proposition}
If \(X\) is polynomial count with counting polynomial

\[
P_X(y)=\sum_k a_k y^k,
\]

then its Hasse--Weil zeta function equals \(\prod_k (1-q^k z)^{-a_k}.\)
    
\end{proposition}

In particular, if \(X\) is a smooth projective variety whose motive is pure Tate, i.e. a direct sum of the form

\[
M(X)
\simeq
\bigoplus_k \mathbb Z(k)[2k]^{a_k},
\]
then the coefficients \(a_k\) are the Betti numbers, $b_{2k}$, and the zeta function factors as:
\[
Z(X,z)
=
\prod_k (1-q^k z)^{-b_{2k}}.
\]

Recall that for a smooth projective complex variety \(X\) of complex dimension \(d\), Poincar\'e duality states that

\[
H^i(X,\mathbb{Q})
\cong
H^{2d-i}(X,\mathbb{Q})^\vee.
\]
Hence, the Betti numbers satisfy

\[
b_i(X)=b_{2d-i}(X).
\]

For the Grassmannian \(\text{Gr}_m^n,\) the complex dimension is \(d = mn.\) Since the cohomology of Grassmannians is concentrated in even degrees, Poincar\'e duality gives \(
b_{2k}
=
b_{2(mn-k)}.
\)
Thus the Poincaré polynomial of the Grassmannian is

\[
\mathcal{P}_t(\text{Gr}_m^n)
=
\sum_{k=0}^{mn} b_{2k} t^{2k},
\]

and one has the classical identity

\[
\mathcal{P}_t(\text{Gr}_m^n)
=
\genfrac{[}{]}{0pt}{}{m+n}{m}_{t^2}.
\]

Therefore Poincar\'e duality implies the symmetry relation

\[
\genfrac{[}{]}{0pt}{}{m+n}{m}_q
=
q^{mn}
\genfrac{[}{]}{0pt}{}{m+n}{m}_{q^{-1}}.
\]

Thus the Gaussian binomial coefficients form a palindromic sequence. Equivalently, if \(p(n,m,k)\) denotes the number of partitions of \(k\) fitting inside an \(m\times n\) rectangle, then

\[
p(n,m,k)
=
p(n,m,mn-k).
\]

Combinatorially, this symmetry corresponds to taking the complement of a Ferrers diagram inside the rectangle.

\begin{figure}[H]
\centering
\includegraphics[width=0.80\textwidth]{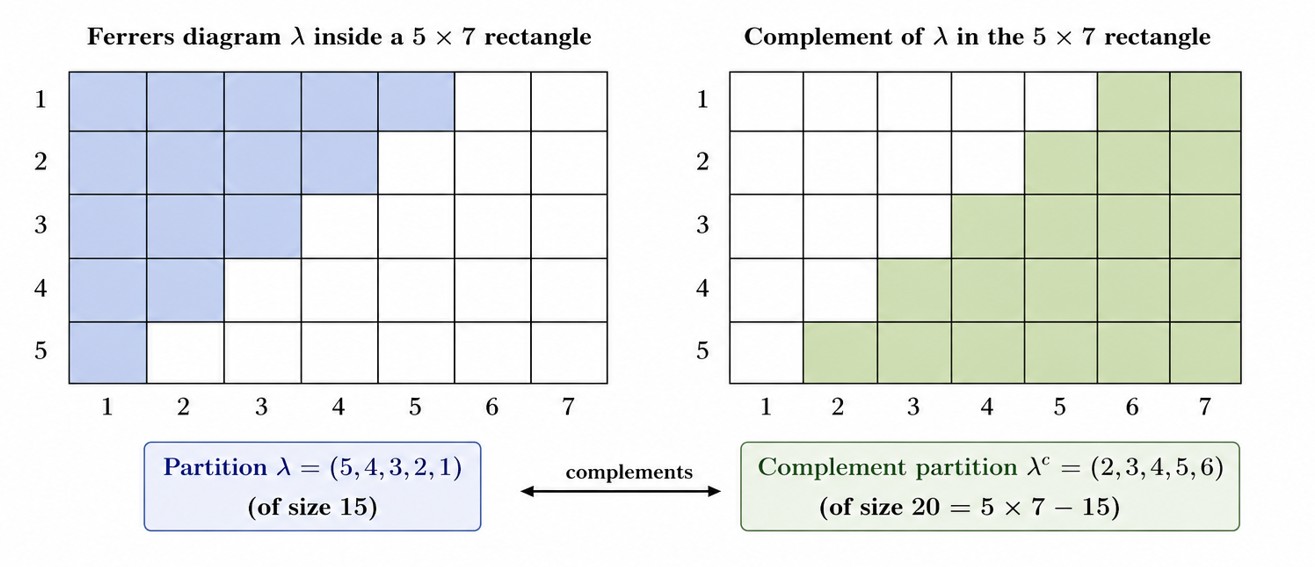}
\caption{A Ferrers diagram is a left-justified array of boxes with non-increasing row lengths. Here \(p(n,m,k)\) denotes the number of partitions of \(k\) fitting inside an \(m\times n\) rectangle. The right diagram become a Ferrers diagram after rotating 180 degrees clockwise.}
\end{figure}

From the motivic viewpoint, however, the symmetry arises from the self-duality of the motive of Grassmannian. Recall that for a smooth projective variety \(X\) of dimension \(d\), one has

\begin{equation}\label{EqPoincareDualityForMotivess}
M(X)^\vee
\cong
M(X)(-d)[-2d].
\end{equation}

\noindent
Hence for the Grassmannian we have

\[
M(\text{Gr}_m^n)^\vee
\cong
M(\text{Gr}_m^n)(-mn)[-2mn],
\]
see \cite[Chapter 5, Theorem 4.3.7]{VSF}.
Let us now consider the dual of both sides of the decomposition given in Proposition \ref{PropositionMotivicKP}, we get:

\[
M(\text{Gr}_m^n)^{\vee}
=
M(\text{Gr}_{m-1}^n)^{\vee}
\oplus
M(\text{Gr}_m^{n-1})^{\vee}(-m)[-2m].
\]

\noindent
This together with the duality (\ref{EqPoincareDualityForMotivess}) gives

\[
M(\text{Gr}_m^n)(-mn)[-2mn]  =  M(\text{Gr}_{m-1}^n)(-n(m-1))[-2n(m-1)]\oplus  M(\text{Gr}_m^{n-1})(-mn)[-2mn]
\]

and finally we obtain the dual decomposition:

\[
M(\text{Gr}_m^n)
=
M(\text{Gr}_{m-1}^n)(n)[2n]
\oplus
M(\text{Gr}_m^{n-1}).
\]

Let us summarize the above discussion. Regarding the above and the discussion in Subsection \ref{SubSectDecomposition}, see Proposition \ref{PropositionMotivicKP}, we conclude that 

\begin{theorem}[Motivic version of Karajī--Pascal Identity]\label{PoincareGrassmannianDecomposition}
There is the following decomposition 

\begin{equation}\label{EqMotivicKhayyam-PascalII}
M(\text{Gr}_m^n)
=
M(\text{Gr}_{m-1}^n)
\oplus
M(\text{Gr}_m^{n-1})(m)[2m],
\end{equation}

and the dual decomposition:
\begin{equation}\label{EqDualMotivicKhayyam-Pascal}
M(\text{Gr}_m^n)
=
M(\text{Gr}_{m-1}^n)(n)[2n]
\oplus
M(\text{Gr}_m^{n-1}).
\end{equation}

\end{theorem}

Note that after applying the \'etale realization functor and computing the trace of Frobenius the decomposition (\ref{EqDualMotivicKhayyam-Pascal}) gives the dual formula in remark \ref{RemDualFormula}. Compare also the discussion in Section \ref{SectLatticePaths}.

\begin{remark}
    The above decompositions, in theorem \ref{PoincareGrassmannianDecomposition}, in particular give decomposition of the corresponding (higher) Chow groups, see Appendix \ref{app:SubSectHigherChow}.
\end{remark}

\begin{example}
Let \(X=\mathrm{Gr}_2^2=\operatorname{Gr}(2,4)\) over \(\mathbb Z\),
and write \(X_k\) for its base change to a field \(k\).
The six lattice paths across the \(2\times2\) square give the following
partitions, obtained by reading the two column heights above each path
and omitting zeros.
\[
  \begin{array}{c|cccccc}
    p & \mathrm{NNEE} & \mathrm{NENE} & \mathrm{ENNE} &
        \mathrm{NEEN} & \mathrm{ENEN} & \mathrm{EENN}\\
    \lambda(p) & \varnothing & (1) & (2) & (1,1) & (2,1) & (2,2)
  \end{array}
\]
Giving a path of area \(a\) the weight \(q^a\) gives
\[
  \genfrac{[}{]}{0pt}{}{4}{2}_q
    =\sum_p q^{|\lambda(p)|}=1+q+2q^2+q^3+q^4.
\]
The coefficient \(2\) counts the partitions \((2)\) and \((1,1)\).

Fix a line \(\ell\subset k^4\). The planes containing \(\ell\) form
\(Z=\mathbb P(k^4/\ell)\simeq\mathbb P^2_k\). Every other plane
projects isomorphically onto a plane \(W\subset k^4/\ell\). After
choosing a complement to \(\ell\), its lifts are the graphs of the
linear maps \(W\to\ell\). Thus \(X_k\setminus Z\) is a rank-two
vector bundle over the dual projective plane
\(B=\operatorname{Gr}(2,k^4/\ell)\simeq\mathbb P^2_k\).

Over a perfect field \(k\), put \(L=\mathbb Z(1)[2]\).
The integral relative cellular decomposition for
\(\varnothing\subset Z\subset X_k\) gives the motivic
Karaj\={\i}--Pascal identity for \(m=n=2\).%
Using \(M(\mathbb P^2_k)\simeq\mathbb Z\oplus L\oplus L^{\otimes2}\)
we obtain
\[
  \begin{aligned}
    M(X_k)
      &\simeq M(\mathbb P^2_k)\oplus M(\mathbb P^2_k)(2)[4]\\
      &\simeq \mathbb Z\oplus L\oplus
        \bigl(L^{\otimes2}\bigr)^{\oplus2}
        \oplus L^{\otimes3}\oplus L^{\otimes4}.
  \end{aligned}
\]
One copy of \(L^{\otimes2}\) is the top summand of \(M(Z)\).
The other comes from the unit summand of \(M(B)\) after applying
\((2)[4]\). For any coefficient ring \(R\), the two middle summands
contribute \(CH^{p-2}(k,r;R)^{\oplus2}\) to \(CH^p(X_k,r;R)\).
In the notation of Appendix~C, the same decomposition gives
\(B_{2,2}^R(x,y)=(1+x+2x^2+x^3+x^4)A_{k,R}(x,y)\) whenever
these dimension series are defined.

Motivic Poincar\'e duality gives \(M(X_k)^\vee(4)[8]\simeq M(X_k)\),
since \(X_k\) is smooth and projective of dimension four. The functor
\(N\mapsto N^\vee(4)[8]\) sends \(L^{\otimes i}\) to \(L^{\otimes(4-i)}\).
Since \(Z\) and \(B\) both have dimension two, it sends \(M(Z)\) to
\(M(Z)(2)[4]\) and \(M(B)(2)[4]\) to \(M(B)\). The dual form of the
Karaj\={\i}--Pascal decomposition is therefore
\[
  M(X_k)\simeq M(Z)(2)[4]\oplus M(B).
\]

 Its six Schubert cells have real dimensions
\(0,2,4,4,6,8\). Their cellular boundary maps vanish, so
the integral cohomology groups are free and the Poincar\'e polynomial is

\[
  \mathcal P_t\bigl(X(\mathbb C)\bigr)
   =\genfrac{[}{]}{0pt}{}{4}{2}_{t^2}
   =1+t^2+2t^4+t^6+t^8.
\]
In particular \(b_4=2\) and \(\chi(X(\mathbb C))=6\). Over \(\mathbb F_q\), each projective plane has \(1+q+q^2\) points,
and every point of \(B\) has \(q^2\) lifts to the open stratum. Hence
\[
  \#X(\mathbb F_q)=(1+q+q^2)(1+q^2),
  \qquad \#X(\mathbb F_2)=7+4\cdot7=35.
\]
Replacing \(q\) by \(q^r\) gives the count over every extension
\(\mathbb F_{q^r}\). The Hasse--Weil zeta function is therefore
\[
  \begin{aligned}
   Z(X_{\mathbb F_q},z)
    &=Z(\mathbb P^2_{\mathbb F_q},z)\,
      Z(\mathbb P^2_{\mathbb F_q},q^2z)\\
    &=\frac{1}{(1-z)(1-qz)(1-q^2z)^2(1-q^3z)(1-q^4z)}.
  \end{aligned}
\]
The two middle Tate summands each contribute the eigenvalue \(q^2\)
of geometric Frobenius in degree four, giving the square in the denominator.
\end{example}

\subsection{Hard Lefschetz and Unimodality}

Let $X$ be a smooth projective complex variety of dimension $d$,
and let
\[
b_k=\dim_{\mathbb Q}H^k(X,\mathbb Q)
\]
denote its Betti numbers.

The Hard Lefschetz theorem states that if
$L\in H^2(X,\mathbb Q)$ is the cohomology class of an
ample divisor, then for every $0\leq k\leq d$,
\[
L^{d-k}:
H^k(X,\mathbb Q)
\longrightarrow
H^{2d-k}(X,\mathbb Q)
\]
is an isomorphism. A consequence is that, for every $k<d$, the map
\[
L:H^k(X,\mathbb Q)
\longrightarrow H^{k+2}(X,\mathbb Q)
\]
is injective. Indeed, if $L\alpha=0$, then
\[
L^{d-k}\alpha
=
L^{d-k-1}(L\alpha)=0.
\]
Since $L^{d-k}$ is an isomorphism, it follows that
$\alpha=0$. Consequently,
\[
b_k\leq b_{k+2}\qquad (k<d).
\]

Thus, the Betti numbers are weakly increasing up to the
middle dimension within each parity. On the other hand,
Poincaré duality gives
\[
b_k=b_{2d-k}.
\]
Combining these two facts shows that the even-degree and
odd-degree Betti numbers form symmetric unimodal sequences
separately.

For Grassmannians, all odd cohomology groups vanish.
Since $\mathrm{Gr}_m^n$ has complex dimension $mn$, one has
\[
\genfrac{[}{]}{0pt}{}{m+n}{m}_q
=
\sum_{k=0}^{mn}b_{2k}q^k.
\]
It follows from Hard Lefschetz that
\[
b_0\leq b_2\leq\cdots\leq
b_{2\lfloor mn/2\rfloor}
=
b_{2\lceil mn/2\rceil}
\geq\cdots\geq b_{2mn},
\]
whereas Poincaré duality gives
\[
b_{2k}=b_{2(mn-k)}.
\]
Therefore, the coefficients of the Gaussian binomial
coefficient form a symmetric unimodal sequence.

Let us now consider the flag variety \(Fl(m,m+1;m+n)\) whose points over a field $L$ parametrize tuples of $L$-vectore spaces $(V,W)$ where $V\subset W\subset L^{m+n}$ 
with $\dim V=m$ and $\dim W=m+1$.
\forget{
\[
Fl(m,m+1;m+n)
=
\left\{
(V,W)\mid V\subset W\subset\mathbb C^{m+n},
\ \dim V=m,\ \dim W=m+1
\right\}.
\]
}
It admits two natural projections:
\[
\begin{aligned}
\pi_1:Fl(m,m+1;m+n)
&\longrightarrow \mathrm{Gr}_m^n,\\
(V,W)&\longmapsto V,
\end{aligned}
\]
and
\[
\begin{aligned}
\pi_2:Fl(m,m+1;m+n)
&\longrightarrow \mathrm{Gr}_{m+1}^{n-1},\\
(V,W)&\longmapsto W.
\end{aligned}
\]

For a fixed $m$-dimensional subspace $V$, the possible
subspaces $W$ correspond to lines in
$\mathbb C^{m+n}/V$. Hence the fibre of $\pi_1$ is
$\mathbb P^{n-1}$. Similarly, for a fixed $(m+1)$-dimensional subspace $W$,
the possible subspaces $V$ are the hyperplanes in $W$.
Thus the fibre of $\pi_2$ is $\mathbb P^m$.

Both projections are projective bundles. By the projective
bundle theorem for motives
\cite[Chapter 5, Proposition~3.5.1]{VSF},
we obtain
\[
M(Fl(m,m+1;m+n))
\cong
M(\mathrm{Gr}_m^n)\otimes M(\mathbb P^{n-1})
\]
and
\[
M(Fl(m,m+1;m+n))
\cong
M(\mathrm{Gr}_{m+1}^{n-1})
\otimes M(\mathbb P^m).
\]
Consequently,
\[
M(\mathrm{Gr}_m^n)\otimes M(\mathbb P^{n-1})
\cong
M(\mathrm{Gr}_{m+1}^{n-1})
\otimes M(\mathbb P^m).
\]

This motivic isomorphism provides a geometric explanation
for the Gaussian binomial identity
\[
\genfrac{[}{]}{0pt}{}{m+n}{m}_q
\frac{q^n-1}{q-1}
=
\genfrac{[}{]}{0pt}{}{m+n}{m+1}_q
\frac{q^{m+1}-1}{q-1}.
\]

Taking Euler characteristics, or equivalently specializing
$q=1$ in the preceding polynomial identity, yields
\[
\binom{m+n}{m}\,n
=
\binom{m+n}{m+1}(m+1).
\]
Equivalently,
\[
\frac{\binom{m+n}{m+1}}
     {\binom{m+n}{m}}
=
\frac{n}{m+1}.
\]

Fixing $N=m+n$, we see that this ratio is strictly
greater than $1$ when
\[
m<\frac{N-1}{2},
\]
equal to $1$ when $N$ is odd and
$m=(N-1)/2$, and strictly less than $1$ when
\[
m>\frac{N-1}{2}.
\]
Consequently, the binomial coefficients
\[
\binom N0,\binom N1,\ldots,\binom NN
\]
form a symmetric unimodal sequence.

Thus, the two projections of the flag variety provide
a geometric explanation for the unimodality of the rows
of the Karajī--Pascal triangle, while Hard Lefschetz
explains the unimodality of the coefficients of its
$q$-analogue.

Here is a final remark:

\begin{remark} 

The guiding principle of motive theory is that many apparently different mathematical structures arise from a single hidden source. 
The Karajī--Pascal identity is only one example of a much broader phenomenon. Many classical identities in combinatorics can be interpreted as realizations of geometric decompositions or filtrations on motives associated with algebraic varieties. Many such identities, and recursive formulas frequently arise from:

\begin{itemize}
    \item[-] stratifications of varieties,
    \item[-] decompositions in triangulated categories,
    \item[-] mixed Hodge structures,
    \item[-] and motivic filtrations.
    
\end{itemize}

From this perspective, combinatorics becomes a shadow of geometry, while motives provide a conceptual framework explaining why many seemingly unrelated identities share the same underlying structure. 
\\
Following a comment by Prof. Kahn, we would like to draw the reader's attention to an interesting related observation on the Beta function and Jacobi sums due to N. Otsubo and T. Yamazaki; see \cite{Ot-Ya}.

\end{remark}

\section*{Acknowledgements} I would like to express my gratitude to Luca Barbieri-Viale and Bruno Kahn for their support and encouragement, especially during the difficult period when the circumstances arising from the war affecting our country created exceptional challenges. This work has been strongly influenced by Arash Rastegar's perspective on mathematics and mathematics education, for which I am deeply grateful. I would also like to thank Esmail Arasteh Rad for his comments and suggestions.

This note was extracted from my talks at a conference at IASBS Zanjan(Aug. 2023), and for that I would like to thank Meysam Nassiri and other organizers, and an additional talk at mathematics institute of TMU (Tehran)(May 2026), and for that I would like to thank the staff and audiences.

\begin{minipage}[t]{0.9\linewidth}
\noindent
\small\textbf{Somayeh Habibi}, School of Mathematics, Institute for Research in Fundamental Sciences (IPM), P.O. Box: 19395-5746, Tehran, Iran
 email: shabibi@ipm.ir

\end{minipage}

\newpage


\appendix

\title{Appendix}

\maketitle

\author{Esmail Arasteh Rad and Somayeh Habibi}

\section{Adequate Equivalence Relations and Their Hierarchy}
\label{app:SubSectAdequateEquivalence}

Let $k$ be a field. Throughout this section, cycles are taken
with rational coefficients:
\[
Z^r(X)_{\mathbb Q}
:=
Z^r(X)\otimes_{\mathbb Z}\mathbb Q.
\]
An equivalence relation on cycles should be compatible with the
operations used to compose correspondences. This requirement is
formalized by the notion of an \emph{adequate equivalence relation};
see \cite[\S4.4]{AppKahnCycles}.

More precisely, an adequate equivalence relation $\sim$ is an equivalence relation on cycles on smooth projective
$k$-varieties satisfying the following conditions:
\begin{enumerate}
\item
It is compatible with the $\mathbb Q$-linear structure and the
grading by codimension.

\item
Given cycles $\alpha$ and $\beta$ on $X$, there exists
$\alpha'\sim\alpha$ such that $\alpha'$ and $\beta$ meet properly.

\item
If $\alpha\sim0$ on $X$ and a cycle $\gamma$ on $X\times Y$
meets $\alpha\times Y$ properly, then
\[
(p_Y)_*
\bigl(
\gamma\cdot(\alpha\times Y)
\bigr)
\sim0
\qquad\text{on }Y.
\]
\end{enumerate}
These conditions ensure that intersection products and the action
of correspondences are well defined modulo $\sim$. We write
\[
A^r_{\sim}(X;\mathbb Q)
:=
Z^r(X)_{\mathbb Q}/\sim.
\]

The principal examples are the following:
\begin{itemize}
\item
\emph{Rational equivalence}, generated by differences of fibres
of families of cycles over $\mathbb P^1$. The resulting quotient is
\[
CH^r(X)_{\mathbb Q}.
\]

\item
\emph{Algebraic equivalence}, obtained by allowing families
parametrized by smooth connected curves rather than only
$\mathbb P^1$.

\item
\emph{Smash-nilpotence equivalence}, introduced by Voevodsky.
A cycle $\alpha\in Z^r(X)_{\mathbb Q}$ is
\emph{smash-nilpotent} if there exists an integer $N\geq1$ such that
\[
\alpha^{\boxtimes N}
:=
p_1^*\alpha\cdots p_N^*\alpha
=
0
\qquad\text{in }
CH^{rN}(X^N)_{\mathbb Q},
\]
where $p_i:X^N\to X$ denotes the $i$-th projection.
Two cycles are smash-equivalent if their difference is
smash-nilpotent. Notice that the external power is taken on
$X^N$, rather than in the Chow ring of $X$ itself.

\item
\emph{Homological equivalence}, relative to a fixed Weil
cohomology theory $H$: two cycles are equivalent if their
cycle classes in $H$ coincide.

\item
\emph{Numerical equivalence}: if $d=\dim X$, a cycle
$\alpha\in Z^r(X)_{\mathbb Q}$ is numerically trivial if
\[
\deg(\alpha\cdot\beta)=0
\qquad
\text{for every }
\beta\in Z^{d-r}(X)_{\mathbb Q},
\]
where the intersections are computed after moving the cycles so
that they meet properly.
\end{itemize}

These adequate equivalence relations satisfy the hierarchy
\[
\sim_{\mathrm{rat}}
\ \Longrightarrow
\sim_{\mathrm{alg}}
\ \Longrightarrow
\sim_{\mathrm{sm}}
\ \Longrightarrow
\sim_{\mathrm{hom},H}
\ \Longrightarrow
\sim_{\mathrm{num}}.
\]
The implication from algebraic equivalence to smash-equivalence
is Voevodsky's nilpotence theorem: every algebraically trivial
cycle with rational coefficients is smash-nilpotent
\cite{AppVoeNil}.

The implication from smash-equivalence to homological equivalence
follows from compatibility of the cycle class map with external
products and the K\"unneth isomorphism. 

Consequently, there are natural surjections
\[
\begin{aligned}
CH^r(X)_{\mathbb Q}
&\twoheadrightarrow
A^r_{\mathrm{alg}}(X;\mathbb Q)
\twoheadrightarrow
A^r_{\mathrm{sm}}(X;\mathbb Q)\\
&\twoheadrightarrow
A^r_{\mathrm{hom},H}(X;\mathbb Q)
\twoheadrightarrow
A^r_{\mathrm{num}}(X;\mathbb Q).
\end{aligned}
\]
Thus rational equivalence is the finest among these adequate
equivalence relations, whereas numerical equivalence is the
coarsest.
The standard conjecture~$D$ predicts that homological and
numerical equivalence coincide for the chosen Weil cohomology
theory. It does not predict that rational and numerical
equivalence coincide. Voevodsky's smash-nilpotence conjecture predicts that every
numerically trivial cycle with rational coefficients is
smash-nilpotent. Equivalently,
\[
\sim_{\mathrm{sm}}
=
\sim_{\mathrm{num}}.
\]
By the preceding hierarchy, this would also imply equality of
homological and numerical equivalence for every Weil cohomology
theory. Thus Voevodsky's conjecture is stronger than the standard
conjecture~$D$.

\begin{remark}
The converse of the first two implications in the above hierarchy are strict in
general, whereas the converses of the last two remain open in
general.

\begin{enumerate}
\item
Algebraic equivalence does not imply rational equivalence.
Let $E/\mathbb C$ be an elliptic curve with origin $O$, and choose
a point $P\in E(\mathbb C)$ of infinite order. The divisor
\[
[P]-[O]
\]
is algebraically trivial, as can be seen from the diagonal family
on $E\times E$, but its class in $CH^1(E)_{\mathbb Q}$ is nonzero.
Indeed, the isomorphism
\[
E(\mathbb C)
\xrightarrow{\sim}
\operatorname{Pic}^0(E),
\qquad
P\longmapsto
\mathcal O_E(P-O),
\]
identifies this class with a point of infinite order; see
\cite[Chapter~I, Proposition~4.10]{MilneEC}.

\item
Smash-equivalence does not imply algebraic equivalence.
Let $C/\mathbb C$ be a very general smooth projective curve of
genus $g\geq3$, embedded in its Jacobian $J(C)$ using a base point.
Its Ceresa cycle
\[
z_C
=
[C]-[-1]_*[C]
\in
CH^{g-1}(J(C))_{\mathbb Q}
\]
is not algebraically trivial by Ceresa's theorem
\cite{Ceresa}. On the other hand,
\[
[-1]^*z_C=-z_C,
\]
so $z_C$ is smash-nilpotent by
\cite[Proposition~1]{KahnSebastian}. In particular, the same
example also shows that homological equivalence does not imply
algebraic equivalence.

\end{enumerate}

\end{remark}

\begin{remark}\label{RemPureTateRatNum}
Let $X/k$ be a smooth projective variety of pure dimension $d$.
Suppose its Chow motive with rational coefficients is pure Tate.
Then rational and numerical equivalence coincide on cycles with
rational coefficients. Indeed, the Tate decomposition gives
\[
[\Delta_X]
=\sum_{a,i}\alpha_{a,i}\times\beta_{a,i}
\quad\text{in }CH^d(X\times X;\mathbb Q),
\]
where
$\alpha_{a,i}\in CH^a(X;\mathbb Q)$ and
$\beta_{a,i}\in CH^{d-a}(X;\mathbb Q)$.
Let $z\in CH^p(X;\mathbb Q)$ be numerically trivial, meaning that
\[
\deg(z\cdot w)=0
\qquad\text{for every }w\in CH^{d-p}(X;\mathbb Q).
\]
Writing $p_1,p_2:X\times X\to X$ for the projections, we obtain
\[
\begin{aligned}
z
&=(p_2)_*\bigl(p_1^*z\cdot[\Delta_X]\bigr)\\
&=\sum_i
  \deg\bigl(z\cdot\alpha_{d-p,i}\bigr)\,
  \beta_{d-p,i}
=0.
\end{aligned}
\]
Note that only $a=d-p$ contributes, since the first factor must be a
zero-cycle to have nonzero push-forward to $\operatorname{Spec}k$.
Thus numerical triviality implies rational triviality. In particular, the intersection pairing
\[
CH^p(X;\mathbb Q)\times CH^{d-p}(X;\mathbb Q)
\longrightarrow\mathbb Q,
\qquad (z,w)\longmapsto\deg(z\cdot w),
\]
is perfect. If the Tate decomposition exists with integral
coefficients, the same argument applies directly to $CH^p(X)$.
\end{remark}

\section{The Karajī--Pascal Identity for Higher Chow Groups}
\label{app:SubSectHigherChow}

Classical Chow groups describe algebraic cycles modulo rational
equivalence and provide an algebro-geometric analogue of homology.
To capture additional geometric and arithmetic information,
Bloch introduced higher Chow groups in the 1980s; see
\cite{Bloch}. His construction replaces individual cycles by
families of cycles parametrized by algebraic simplices and
organizes them into a chain complex. The ordinary Chow groups
occur in degree zero,
\[
CH^p(X,0)=CH^p(X),
\]
while the higher groups contain finer information. Bloch's higher
Chow groups were subsequently identified with motivic cohomology,
thereby providing an explicit bridge between algebraic cycles,
algebraic $K$-theory, and motives.

\begin{definition}[Bloch]
Let $X$ be a quasi-projective variety over a field $k$. For
integers $p,n\geq0$, let $z^p(X,n)$ denote the free abelian group
generated by integral closed subvarieties
\[
Z\subset X\times\Delta^n
\]
of codimension $p$ that meet every face
\[
X\times\Delta^m
\hookrightarrow
X\times\Delta^n
\]
properly. Here
\[
\Delta^n
=
\operatorname{Spec}
k[t_0,\ldots,t_n]/(t_0+\cdots+t_n-1)
\]
is the algebraic $n$-simplex. The face maps induce a boundary
operator
\[
\partial
=
\sum_{i=0}^n(-1)^i\partial_i:
z^p(X,n)
\longrightarrow
z^p(X,n-1),
\]
making $z^p(X,\bullet)$ into a chain complex.

Let $R$ be a commutative coefficient ring. The \emph{higher Chow
groups with coefficients in $R$} are
\[
CH^p(X,n;R)
:=H_n\bigl(z^p(X,\bullet)\otimes_{\mathbb Z}R\bigr).
\]
We write $CH^p(X,n):=CH^p(X,n;\mathbb Z)$ and set
$CH^p(X,n;R)=0$ if $p<0$ or $n<0$. In degree zero,
\[
CH^p(X,0;R)=CH^p(X)\otimes_{\mathbb Z}R.
\]
\end{definition}

Coefficients are thus introduced at the level of the cycle complex.
Since its terms are free abelian groups, the universal-coefficient
sequence gives
\[
0\longrightarrow CH^p(X,n)\otimes_{\mathbb Z}R
\longrightarrow CH^p(X,n;R)
\longrightarrow
\operatorname{Tor}^{\mathbb Z}_1\bigl(CH^p(X,n-1),R\bigr)
\longrightarrow0.
\]
If $R$ is flat over $\mathbb Z$, the last term vanishes. In
particular, coefficients in a field of characteristic zero agree
with tensoring the integral groups with that field. For coefficients
in $\mathbb F_\ell$, however, the last term is
$CH^p(X,n-1)[\ell]$, where $G[\ell]$ denotes the subgroup of
$G$ annihilated by $\ell$. This contribution must be retained.

For a smooth variety $X$, the comparison with motivic cohomology is
\[
CH^p(X,r;R)
\cong H^{2p-r}_{\mathrm M}(X,R(p))
\cong
\operatorname{Hom}_{DM(k,R)}
\bigl(M(X;R),R(p)[2p-r]\bigr);
\]
see \cite{AppVSF}. For the Grassmannians considered here, the integral
motivic decompositions extend to $R$-coefficients. Applying the
preceding contravariant Hom functor to
\[
M(\mathrm{Gr}_m^n;R)
\cong M(\mathrm{Gr}_{m-1}^n;R)
\oplus M(\mathrm{Gr}_m^{n-1};R)(m)[2m]
\]
gives
\[
CH^p(\mathrm{Gr}_m^n,r;R)
\cong CH^p(\mathrm{Gr}_{m-1}^n,r;R)
\oplus CH^{p-m}(\mathrm{Gr}_m^{n-1},r;R).
\]
Indeed, the twist and shift replace $(p,2p-r)$ by
$(p-m,2(p-m)-r)$, so the higher-Chow degree $r$ is unchanged.
Similarly, the dual decomposition
\[
M(\mathrm{Gr}_m^n;R)
\cong M(\mathrm{Gr}_{m-1}^n;R)(n)[2n]
\oplus M(\mathrm{Gr}_m^{n-1};R)
\]
yields
\[
CH^p(\mathrm{Gr}_m^n,r;R)
\cong CH^{p-n}(\mathrm{Gr}_{m-1}^n,r;R)
\oplus CH^p(\mathrm{Gr}_m^{n-1},r;R).
\]
These are the higher-Chow counterparts of the two dual
Karajī--Pascal recursions.

\section{Generating Functions for Higher Chow Groups and the Contribution of the Base Field}
\label{app:SubSectGeneratingFunctions}

Let $R$ be a coefficient field, and suppose that the base field
$k$ satisfies
\[
\dim_R
\bigl(
CH^j(k,r;R)
\bigr)
<\infty
\qquad
\text{for all }j,r\geq0,
\]
as a vector space over $R$, where
\[
CH^j(k,r;R)
:=
CH^j(\operatorname{Spec}k,r;R).
\]
We use the standard identification of higher Chow groups with
motivic cohomology; see \cite[Theorem~19.1]{MVW}.
For $m,n\geq0$, define
\[
b_{m,n}^R(p,r)
:=
\dim_R
\bigl(
CH^p(\mathrm{Gr}_m^n,r;R)
\bigr)
\]
and the corresponding generating series
\[
B_{m,n}^R(x,y)
:=
\sum_{p,r\geq0}
b_{m,n}^R(p,r)x^py^r.
\]
Since $\mathrm{Gr}_m^n$ is cellular, its higher Chow groups
decompose into finitely many shifted copies of the higher Chow
groups of the base field. Hence the above dimensions are finite,
and
\[
B_{m,n}^R(x,y)
\in
\mathbb Z[[x,y]].
\]

The two decompositions obtained from Theorem~\ref{PoincareGrassmannianDecomposition} give

\begin{equation}\label{EqCHKP1}
CH^p(\mathrm{Gr}_m^n,r;R)
\cong
CH^p(\mathrm{Gr}_{m-1}^n,r;R)
\oplus
CH^{p-m}(\mathrm{Gr}_m^{n-1},r;R)
\end{equation}
and
\begin{equation}\label{EqCHKP2}
CH^p(\mathrm{Gr}_m^n,r;R)
\cong
CH^{p-n}(\mathrm{Gr}_{m-1}^n,r;R)
\oplus
CH^p(\mathrm{Gr}_m^{n-1},r;R),
\end{equation}
which yield the following higher-Chow analogues of the two dual
Karajī--Pascal identities.

\begin{proposition}[Higher-Chow Karajī--Pascal identities]
\label{app:PropHigherChowKarajiPascal}
For $m,n\geq1$, one has
\[
B_{m,n}^R(x,y)
=
B_{m-1,n}^R(x,y)
+
x^m B_{m,n-1}^R(x,y)
\]
and
\[
B_{m,n}^R(x,y)
=
x^nB_{m-1,n}^R(x,y)
+
B_{m,n-1}^R(x,y).
\]
If
\[
A_{k,R}(x,y)
:=
\sum_{j,r\geq0}
\dim_R
\bigl(
CH^j(k,r;R)
\bigr)
x^jy^r,
\]
then
\[
B_{0,n}^R(x,y)
=
B_{m,0}^R(x,y)
=
A_{k,R}(x,y),
\]
and
\[
B_{m,n}^R(x,y)
=
\genfrac{[}{]}{0pt}{}{m+n}{m}_x
A_{k,R}(x,y).
\]
\end{proposition}

\begin{proof}
Taking $R$-dimensions in \eqref{EqCHKP1} and \eqref{EqCHKP2}
and summing with weights $x^py^r$ gives the two recursions.
The boundary conditions follow from
$\mathrm{Gr}_0^n=\mathrm{Gr}_m^0=\operatorname{Spec}k$.
The factorization follows by induction on $m+n$ using either
recursion and the corresponding Gaussian binomial recursion.
\end{proof}

The variable $y$ records the higher-Chow degree, which is
unchanged in both decompositions, whereas the powers of $x$
record the shifts in codimension. Moreover,
\[
A_{k,R}(x,0)=1,
\]
since
\[
CH^0(k,0;R)=R,
\qquad
CH^j(k,0;R)=0
\quad
(j>0).
\]
Consequently,
\[
B_{m,n}^R(x,0)
=
\genfrac{[}{]}{0pt}{}{m+n}{m}_x.
\]
Thus the ordinary Chow groups recover the usual dual
$q$-Karajī--Pascal identities, with the formal parameter $q$
replaced by $x$:
\[
\genfrac{[}{]}{0pt}{}{m+n}{m}_x
=
x^n
\genfrac{[}{]}{0pt}{}{m+n-1}{m-1}_x
+
\genfrac{[}{]}{0pt}{}{m+n-1}{m}_x
\]
and
\[
\genfrac{[}{]}{0pt}{}{m+n}{m}_x
=
\genfrac{[}{]}{0pt}{}{m+n-1}{m-1}_x
+
x^m
\genfrac{[}{]}{0pt}{}{m+n-1}{m}_x.
\]

\paragraph{The normalized generating function.}
Since $A_{k,R}(x,y)$ has constant term $1$, it is invertible in
$\mathbb Z[[x,y]]$. It is therefore natural to define
\[
\widetilde B_{m,n}^R(x,y)
:=
\frac{B_{m,n}^R(x,y)}
     {A_{k,R}(x,y)}.
\]
The preceding proposition immediately gives
\[
\widetilde B_{m,n}^R(x,y)
=
\genfrac{[}{]}{0pt}{}{m+n}{m}_x.
\]
\forget{
-----------------------
In particular, the normalized series is independent of $y$, of
the coefficient field $R$, and of the base field $k$. It satisfies
\[
\widetilde B_{m,n}^R
=
x^n\widetilde B_{m-1,n}^R
+
\widetilde B_{m,n-1}^R
\]
and
\[
\widetilde B_{m,n}^R
=
\widetilde B_{m-1,n}^R
+
x^m\widetilde B_{m,n-1}^R,
\]
with boundary conditions
\[
\widetilde B_{0,n}^R
=
\widetilde B_{m,0}^R
=
1.
\]
--------------------
}
Thus the factorization
\[
B_{m,n}^R
=
\widetilde B_{m,n}^R A_{k,R}
\]
separates the universal contribution of the Grassmannian from
the higher-Chow information carried by the base field.

\paragraph{Finite base fields.}
Let $k=\mathbb F_q$, and let $R$ be a coefficient field.
Quillen's computation of the $K$-theory of finite fields
\cite{QuillenFinite}, together with Voevodsky's motivic comparison
theorems \cite[Theorems~6.17--6.18]{VoeFiniteCoefficients}, gives
\[
CH^j(\mathbb F_q,r)
\cong
\begin{cases}
\mathbb Z,&(j,r)=(0,0),\\
\mathbb Z/(q^j-1),&j\geq1,\ r=2j-1,\\
0,&\text{otherwise},
\end{cases}
\]
see also the characteristic-primary calculation of Geisser--Levine
\cite[Theorem~1.1]{GeisserLevine}. In particular, the finiteness hypothesis holds for every $R$.
If $\operatorname{char}R=0$, all the positive-weight groups vanish
with $R$-coefficients, so $A_{\mathbb F_q,R}(x,y)=1$.
The same equality holds if
$\operatorname{char}R=\operatorname{char}\mathbb F_q$, since
$q^j-1$ is prime to that characteristic.

Suppose now that $\ell=\operatorname{char}R>0$ and $\ell\nmid q$.
The universal-coefficient sequence
\cite[Lecture~4]{MVW} gives
\[
\dim_R CH^j(\mathbb F_q,r;R)
=
\begin{cases}
1,&(j,r)=(0,0),\\
1,&j\geq1,\ \ell\mid q^j-1,\ r\in\{2j-1,2j\},\\
0,&\text{otherwise}.
\end{cases}
\]
The group in degree $2j-1$ comes from tensoring the integral
cyclic group with $R$, while the group in degree $2j$ comes
from its $\operatorname{Tor}_1$ contribution. Both are required
for higher Chow groups with coefficients.

Set $d=\operatorname{ord}_\ell(q)$, the multiplicative order of
$q$ in $\mathbb F_\ell^\times$. Since
$\ell\mid q^j-1$ if and only if $d\mid j$, we obtain
\[
\begin{aligned}
A_{\mathbb F_q,R}(x,y)
&=1+\sum_{a\geq1}x^{ad}
\bigl(y^{2ad-1}+y^{2ad}\bigr)\\
&=1+\frac{x^dy^{2d-1}(1+y)}{1-x^dy^{2d}}.
\end{aligned}
\]
Consequently,
\[
B_{m,n}^R(x,y)
=
\genfrac{[}{]}{0pt}{}{m+n}{m}_x
\left(1+\frac{x^dy^{2d-1}(1+y)}{1-x^dy^{2d}}\right).
\]

\forget{
====================================
\paragraph{Arithmetic interpretation of the normalized factor.}
Although
\[
\widetilde B_{m,n}^R(x,y)
=
\genfrac{[}{]}{0pt}{}{m+n}{m}_x
\]
is independent of the base field, over a finite field its
specializations admit a direct arithmetic interpretation. Write
\[
\widetilde B_{m,n}^R(x,y)
=
\sum_{a=0}^{mn}
c_{m,n}(a)x^a.
\]
The integers $c_{m,n}(a)$ are the Schubert multiplicities, or
equivalently the numbers of partitions of $a$ contained in an
$m\times n$ rectangle.

For every $s\geq1$,
\[
\#
\mathrm{Gr}_m^n(\mathbb F_{q^s})
=
\widetilde B_{m,n}^R(q^s)
=
\genfrac{[}{]}{0pt}{}{m+n}{m}_{q^s}.
\]

Thus the normalized generating function determines all point
counts over finite extensions of $\mathbb F_q$. Consequently, it
also determines the Hasse--Weil zeta function:
\[
\begin{aligned}
Z(\mathrm{Gr}_m^n/\mathbb F_q,t)
&=
\exp\left(
\sum_{s\geq1}
\#
\mathrm{Gr}_m^n(\mathbb F_{q^s})
\frac{t^s}{s}
\right)\\
&=
\prod_{a=0}^{mn}
(1-q^at)^{-c_{m,n}(a)}.
\end{aligned}
\]
The same coefficients describe the Tate summands and the
Frobenius eigenvalues in $\ell'$-adic cohomology:
\[
H_{\mathrm{\acute et}}^{2a}
\bigl(
\mathrm{Gr}_m^n{}_{\overline{\mathbb F}_q},
\mathbb Q_{\ell'}
\bigr)
\cong
\mathbb Q_{\ell'}(-a)^{\oplus c_{m,n}(a)},
\qquad
H_{\mathrm{\acute et}}^{2a+1}=0,
\]
for $\ell'\neq\operatorname{char}\mathbb F_q$.
=============================
}

Thus, over a finite field, the factorization
\[
B_{m,n}^R
=
\underbrace{\widetilde B_{m,n}^R}
_{\text{Schubert--Tate geometry}}
\,
\underbrace{A_{\mathbb F_q,R}}
_{\text{arithmetic of the base field}}
\]
separates two distinct kinds of information. The first factor
records the Schubert multiplicities and, after specialization,
the Frobenius weights, point counts, and zeta function of the
Grassmannian. The second detects the $\ell$-primary motivic and
$K$-theoretic arithmetic of the base field.

\begin{remark}\label{rem:elliptic-higher-chow-series}
The normalization removes all $y$-dependence in the pure Tate
case, but need not do so in general. To clarify this we provide the following examples.
Let $E/k$ be an elliptic curve and let $R$ be a coefficient
field. Whenever the relevant dimensions are finite, put
\[
B_{E,R}(x,y)
:=
\sum_{a,r\geq0}
\dim_R CH^a(E,r;R)\,x^ay^r,
\qquad
\widetilde B_{E,R}(x,y)
:=
\frac{B_{E,R}(x,y)}{A_{k,R}(x,y)}.
\]
In particular, finite coefficients include the
universal-coefficient torsion contributions.

We have the following Chow--K\"unneth decomposition
\[
M(E;R)
\cong
R\oplus M_1(E;R)\oplus R(1)[2].
\]

Here

\[
M_1(E;R)
:=\operatorname{im}\bigl(
\pi_1:M(E;R)\longrightarrow M(E;R)\bigr),
\]
corresponding to the correspondence
\[
\pi_1=[\Delta_E]-[E\times\{o\}]-[\{o\}\times E].
\]
It is induced by integral Chow correspondences and hence
remains valid after derived extension to any $R$. The two Tate summands contribute
$(1+x)A_{k,R}(x,y)$ to $B_{E,R}(x,y)$.
The contribution of $M_1(E;R)$ may retain dependence on $y$
after normalization.

\forget{
Suppose $\ell\neq\operatorname{char}k$.
For weight $a$, the comparison map
\[
H^i_{\mathrm M}(E,\mathbb F_\ell(a))
\longrightarrow
H^i_{\mathrm{\acute et}}(E,\mathbb F_\ell(a))
\]
is an isomorphism for $i\leq a$ and is injective for
$i=a+1$; see
\cite[Theorem~6.17]{VoeFiniteCoefficients}.
For $i>a+1$, putting $r=2a-i$ gives $r<a-1$.
If $r\geq0$, then $\dim(E\times\Delta^r)=r+1<a$, so
$z^a(E,r)=0$; groups with $r<0$ vanish by convention.
These observations, together with the injection at the boundary,
will account for degrees outside the isomorphism range.
}

Suppose first that $k$ is algebraically closed and
$\operatorname{char}R=\ell>0$, with
$\ell\neq\operatorname{char}k$.
Motivic comparison, see
\cite[Theorem~6.17]{VoeFiniteCoefficients}, together with the weight-one calculation,
gives
\[
\dim_R CH^a(E,r;R)
=
\begin{cases}
1,&r=2a,\\
2,&r=2a-1,\\
1,&r=2a-2,\\
0,&\text{otherwise},
\end{cases}
\qquad a\geq1;
\]
see \cite[Corollary~4.9]{MVW} and
\cite[Theorem~6.17]{VoeFiniteCoefficients}.
In weight zero, only $CH^0(E,0;R)=R$ survives.
Consequently,
\[
\begin{aligned}
B_{E,R}(x,y)
&=
1+\sum_{a\geq1}x^a
\bigl(y^{2a}+2y^{2a-1}+y^{2a-2}\bigr)\\
&=
\frac{1+x+2xy}{1-xy^2}.
\end{aligned}
\]
Since $A_{k,R}(x,y)=(1-xy^2)^{-1}$, we obtain \(\widetilde B_{E,R}(x,y)=1+x+2xy.\) Thus we see that the contribution of $M_1(E;R)$ produces the $y$-dependence. Now let $k=\mathbb F_q$ have characteristic $p$, and suppose
$\operatorname{char}R=\ell\neq p$.
Set $d=\operatorname{ord}_{\ell}(q)$, and let $\varphi_q$ denote
the $q$-power Frobenius endomorphism acting on
$E[\ell](\overline{\mathbb F}_q)$.
For $j\geq1$, put
\[
s_j
:=
\dim_{\mathbb F_\ell}
\ker\!\left(
q^{j-1}\varphi_q-\mathrm{id}
\mid E[\ell](\overline{\mathbb F}_q)
\right).
\]
Then $s_{j+d}=s_j$, and \(s_1=\dim_{\mathbb F_\ell}E(\mathbb F_q)[\ell].\) The Galois module
\[
H^1_{\mathrm{\acute et}}
\bigl(E_{\overline{\mathbb F}_q},\mathbb F_\ell(j)\bigr)
\cong E[\ell](j-1)
\]
has Frobenius invariants and coinvariants of dimension $s_j$.
Motivic comparison and the Hochschild--Serre spectral sequence \cite[Chap.~III, Thm.~2.20 and Rem.~2.21]{MilneEtCoh}, 
with weight one treated by the Kummer sequence, give
\[
B_{E,R}(x,y)
=
(1+x)A_{\mathbb F_q,R}(x,y)
+
(1+y)\sum_{j\geq1}s_jx^jy^{2j-2}.
\]
See
\cite[II, 2.18(b), and III, Proposition~4.11]{MilneEtCoh}
and the description of weight-one motivic cohomology in
\cite[Corollary~4.9]{MVW}. Using the periodicity of $s_j$ and the finite-field formula for
$A_{\mathbb F_q,R}$, this becomes
\[
B_{E,R}(x,y)
=
\frac{
(1+x)(1+x^dy^{2d-1})
+
(1+y)\displaystyle\sum_{j=1}^{d}s_jx^jy^{2j-2}
}{
1-x^dy^{2d}
}.
\]
Hence
\[
\widetilde B_{E,R}(x,y)
=
1+x+
\frac{
(1+y)\displaystyle\sum_{j=1}^{d}s_jx^jy^{2j-2}
}{
1+x^dy^{2d-1}
}.
\]
The normalized series thus depends on the Frobenius action on
the $\ell$-torsion, and is independent of $y$ precisely when
$s_1=\cdots=s_d=0$.

If $\operatorname{char}R=p$, put $t_p=\dim_{\mathbb F_p}E(\mathbb F_q)[p]$. Then
\[
B_{E,R}(x,y)
=
\widetilde B_{E,R}(x,y)
=
1+(1+t_p)x+t_pxy.
\]
Indeed, the weight-one groups follow from
$\operatorname{Pic}(E)\cong\mathbb Z\oplus E(\mathbb F_q)$ and
the universal-coefficient sequence, while the groups of
codimension at least two vanish by
\cite[Theorem~8.3]{GeisserLevine}, since the corresponding
logarithmic differential forms vanish on a smooth curve.

\end{remark}

\paragraph{The base field $\mathbb Q_p$.}
We first explain why characteristic-zero coefficients do not
satisfy the finiteness hypothesis. The weight-one computation
\cite[Corollary~4.2]{MVW} gives
\[
CH^1(\mathbb Q_p,1;\mathbb Q)
=\mathbb Q_p^\times\otimes_{\mathbb Z}\mathbb Q.
\]
For odd $p$, the valuation and Teichmüller decomposition yield
\[
\mathbb Q_p^\times
\cong p^{\mathbb Z}\times\mu_{p-1}\times(1+p\mathbb Z_p),
\]
and the $p$-adic logarithm identifies $1+p\mathbb Z_p$ with the
additive group $p\mathbb Z_p$. For $p=2$, one instead uses
\[
\mathbb Q_2^\times
\cong 2^{\mathbb Z}\times\{\pm1\}\times(1+4\mathbb Z_2),
\]
and $\log(1+4\mathbb Z_2)=4\mathbb Z_2$; see
\cite[Chapter~II, subsections ~3.1--3.2]{SerreArithmetic}.
Thus, as abelian groups,
\[
\mathbb Q_p^\times\otimes_{\mathbb Z}\mathbb Q
\cong\mathbb Q\oplus\mathbb Q_p.
\]
Here $\mathbb Z_p\otimes_{\mathbb Z}\mathbb Q=\mathbb Q_p$,
and $\mathbb Q_p$ has infinite dimension over $\mathbb Q$.
Consequently $CH^1(\mathbb Q_p,1;R)$ is infinite-dimensional over
every coefficient field $R$ of characteristic zero. The numerical
series $A_{\mathbb Q_p,R}$ is therefore unavailable in that case. 

Now let $R$ be a coefficient field of characteristic $\ell>0$.
Because $R$ is flat over its prime field, there are natural
isomorphisms
\[
CH^j(\mathbb Q_p,r;R)
\cong CH^j(\mathbb Q_p,r;\mathbb F_\ell)
\otimes_{\mathbb F_\ell}R.
\]
It suffices to compute dimensions over $\mathbb F_\ell$; the
resulting series is unchanged by extending this coefficient field.
In particular, we use the same notation $A_{\mathbb Q_p,R}$ as
above, without introducing a separate finite-coefficient series.

Put $G_p=\operatorname{Gal}(\overline{\mathbb Q}_p/\mathbb Q_p)$
and write $\mathbb F_\ell(j)=\mu_\ell^{\otimes j}$, viewed as a
one-dimensional continuous $G_p$-module; negative twists are
defined by tensor powers of the dual module. Since
$\operatorname{char}\mathbb Q_p=0$, the prime $\ell$ is invertible
in the base field, including when $\ell=p$. The
Beilinson--Lichtenbaum comparison, a consequence of the
norm-residue theorem, therefore applies in both cases
\cite[Theorem~6.17]{VoeFiniteCoefficients}. For $j\geq1$, it gives
\[
CH^j(\mathbb Q_p,r;\mathbb F_\ell)
\cong
H^{2j-r}(G_p,\mathbb F_\ell(j))
\qquad\text{if }2j-r\leq j.
\]
If $r<j$, the higher Chow group vanishes, thus there are no
codimension-$j$ cycles on the $r$-dimensional algebraic simplex.
Together with the comparison, this also shows vanishing for
$r>2j$, since negative Galois cohomology is zero.

The $\ell$-cohomological dimension of $G_p$ is $2$;
see \cite[Chapter~I, \S~2]{MilneADT}. Set
\[
h_j^i:=\dim_{\mathbb F_\ell}
H^i(G_p,\mathbb F_\ell(j)),\qquad i=0,1,2.
\]
These dimensions are finite by local Galois cohomology
\cite[Chapter~I, Corollary~2.3]{MilneADT}. Thus positive-weight higher
Chow groups can occur only on the three diagonals
\[
r=2j,\qquad r=2j-1,\qquad r=2j-2,
\]
with the last diagonal allowed only for $j\geq2$.
In weight zero, only $CH^0(\mathbb Q_p,0;R)=R$ survives.
Consequently,
\begin{equation}\label{EqLocalChowSeries}
A_{\mathbb Q_p,R}(x,y)
=1+\sum_{j\geq1}h_j^0x^jy^{2j}
+\sum_{j\geq1}h_j^1x^jy^{2j-1}
+\sum_{j\geq2}h_j^2x^jy^{2j-2}.
\end{equation}
The restriction $j\geq2$ in the last sum is essential:
$H^2(G_p,\mathbb F_\ell(1))$ is nonzero, but lies outside the
motivic comparison range for weight $1$ and does not contribute
$CH^1(\mathbb Q_p,0;\mathbb F_\ell)$.

\emph{Coefficients of characteristic $\ell\neq p$.}
Let $d=\operatorname{ord}_\ell(p)$. The mod-$\ell$ cyclotomic
character is unramified, and arithmetic Frobenius acts on
$\mu_\ell$ by $\zeta\mapsto\zeta^p$. Its image has order $d$.
Hence
\[
h_j^0=\mathbf 1_{d\mid j},
\]
where $\mathbf 1_P$ equals $1$ if the condition $P$ holds and
$0$ otherwise. Local Tate duality identifies
\[
H^2(G_p,\mathbb F_\ell(j))
\cong H^0(G_p,\mathbb F_\ell(1-j))^\vee,
\]
where $(-)^\vee$ denotes the $\mathbb F_\ell$-linear dual; see
\cite[Chapter~I, Corollary~2.3]{MilneADT}.
It follows that
\[
h_j^2=\mathbf 1_{d\mid(j-1)}.
\]
Finally, the local Euler characteristic formula gives
\[
h_j^0-h_j^1+h_j^2=0,
\qquad
h_j^1=\mathbf 1_{d\mid j}+\mathbf 1_{d\mid(j-1)};
\]
see \cite[Chapter~I, Theorem~2.8]{MilneADT}.

Substituting in \eqref{EqLocalChowSeries}, we obtain
\[
\begin{aligned}
\sum_{j\geq1}h_j^0x^jy^{2j}
&=\frac{x^dy^{2d}}{1-x^dy^{2d}},\\
\sum_{j\geq1}h_j^1x^jy^{2j-1}
&=\frac{x^dy^{2d-1}+xy}{1-x^dy^{2d}},\\
\sum_{j\geq2}h_j^2x^jy^{2j-2}
&=\frac{x^{d+1}y^{2d}}{1-x^dy^{2d}}.
\end{aligned}
\]
Therefore
\begin{equation}\label{EqLocalChowAwayP}
A_{\mathbb Q_p,R}(x,y)
=1+\frac{xy+x^dy^{2d-1}+x^dy^{2d}
+x^{d+1}y^{2d}}{1-x^dy^{2d}}.
\end{equation}
When $d=1$, the two contributions to $xy$ are both present,
so its coefficient is $2$. Recall that in  Kummer
isomorphism
\[
CH^1(\mathbb Q_p,1;\mathbb F_\ell)
\cong\mathbb Q_p^\times/(\mathbb Q_p^\times)^\ell:
\]
the valuation contributes one dimension, and the roots of unity
contribute another exactly when $\ell\mid p-1$. The principal
units contribute none, since raising to the $\ell$-th power is an
automorphism on a sufficiently small principal-unit subgroup.

\emph{Coefficients of characteristic $\ell=p$.}
The image of the mod-$p$ cyclotomic character has order $e=p-1$
(with $e=1$ for $p=2$). Thus invariants and local duality give
\[
h_j^0=\mathbf 1_{e\mid j},
\qquad
h_j^2=\mathbf 1_{e\mid(j-1)}.
\]
The Euler characteristic formula
\cite[Chapter~I, Theorem~2.8]{MilneADT} now reads
\[
h_j^0-h_j^1+h_j^2=-1.
\]
Therefore
\[
h_j^1=1+\mathbf 1_{e\mid j}
+\mathbf 1_{e\mid(j-1)}.
\]
The periodic contributions are summed as before, and the extra
$1$ contributes
\[
\sum_{j\geq1}x^jy^{2j-1}=\frac{xy}{1-xy^2}.
\]
Consequently,
\begin{equation}\label{EqLocalChowAtP}
A_{\mathbb Q_p,R}(x,y)
=1+\frac{xy+x^ey^{2e-1}+x^ey^{2e}
+x^{e+1}y^{2e}}{1-x^ey^{2e}}
+\frac{xy}{1-xy^2}.
\end{equation}
In particular, the coefficient of $xy$ is $2$ for odd $p$ and
$3$ for $p=2$. The Kummer description explains these numbers:
for odd $p$, valuation and principal units each contribute one
dimension; for $p=2$, the additional factor $\{\pm1\}$ contributes
one more. In higher weights, the extra diagonal term is governed
by the same local Euler characteristic formula.

\forget{
========================================
Equations~\eqref{EqLocalChowAwayP} and \eqref{EqLocalChowAtP}
show that every coefficient is finite. Thus the generating
functions and their normalization are well defined over
$\mathbb Q_p$ for every coefficient field of positive
characteristic, and
\[
B_{m,n}^R(x,y)
=\genfrac{[}{]}{0pt}{}{m+n}{m}_x A_{\mathbb Q_p,R}(x,y),
\qquad
\widetilde B_{m,n}^R(x,y)
=\genfrac{[}{]}{0pt}{}{m+n}{m}_x.
\]
The Gaussian polynomial records the Schubert multiplicities.
The base-field series records the dimensions of local Galois
cohomology in the motivic range: cyclotomic invariants on the
first diagonal, their Tate-dual counterparts on the third, and
the Euler-characteristic contribution on the middle diagonal.

==============================
}

\begin{minipage}[t]{0.9\linewidth}
\noindent
\small\textbf{Esmail Arasteh Rad}, School of Mathematics, Institute for Research in Fundamental Sciences (IPM), P.O. Box: 19395-5746, Tehran, Iran 
email: earasteh@ipm.ir\\
\\[1mm]

\noindent
\small\textbf{Somayeh Habibi}, School of Mathematics, Institute for Research in Fundamental Sciences (IPM), P.O. Box: 19395-5746, Tehran, Iran
 email: \href{shabibi@ipm.ir}{shabibi@ipm.ir}

\end{minipage}

\end{document}